\documentclass{article}
\usepackage{amsthm}
\theoremstyle{definition}
\newtheorem{definition}{Definition}

\theoremstyle{theorem}
\newtheorem{fact}{Fact}
\newtheorem{proposition}{Proposition}
\newtheorem{lemma}{Lemma}
\newtheorem{theorem}{Theorem}
\newtheorem{corollary}{Corollary}
\theoremstyle{definition}

\pdfoutput=1

\newcommand{\cf}{cf.\ }
\newcommand{\st}{s.t.\ }

\usepackage{xspace}
\newcommand{\etc}{etc.\xspace}
\newcommand{\eg}{e.g.,\xspace}
\newcommand{\ie}{i.e.,\xspace}

\usepackage{array}
\usepackage[pdftex,dvipsnames]{color}
\usepackage{ifthen,calc,url}

\usepackage{amstext,amsmath,amssymb,stmaryrd,mathrsfs}
\usepackage[all,cmtip]{xy}
\usepackage{rotating}

\newcommand{\defeq}{\mathrel{:=}}
\newcommand{\defiff}{\mbox{:iff}}
\newcommand{\bnfeq}{\mathrel{::=}}
\newcommand{\bnfor}{\ \big| \ }

\newcommand{\set}[1]{\{#1\}}
\newcommand{\setst}[2]{\{\ #1\ \boldsymbol{|}\ #2\ \}}

\newcommand{\powerset}[1]{2^{#1}}

\newcommand{\nn}[1]{\makebox[2em][r]{#1}\ \ }
\newcommand{\true}{\text{true}}
\newcommand{\false}{\text{false}}

\newcommand{\agents}{\mathcal{A}}

\newcommand{\community}{\mathcal{C}}

\newcommand{\messages}{\mathcal{M}}
\newcommand{\data}{\mathcal{D}}

\newcommand{\pair}[2]{(#1,#2)}
\newcommand{\sign}[2]{{\{\negmedspace[#1]\negmedspace\}}_{#2}}

\newcommand{\states}{\mathcal{S}}

\newcommand{\gscc}[2]{\mathtt{succ}_{#1}^{#2}}
\newcommand{\msgs}[1]{\mathrm{msgs}_{#1}}
\newcommand{\preorder}[1]{\leq_{#1}}
\newcommand{\indist}[3]{#2\equiv_{#1}#3}

\newcommand{\pAccess}[3]{\mathrel{_{#1}\negthinspace\mathrm{R}_{#2}^{#3}}}
\newcommand{\access}[3]{\mathrel{_{#1}\negthinspace\mathcal{R}_{#2}^{#3}}}

\newcommand{\clo}[2]{\mathrm{cl}_{#1}^{#2}}
\newcommand{\Clo}[2]{\mathrm{Cl}_{#1}^{#2}}
\newcommand{\LiP}{\mathrm{LiP}}
\newcommand{\pFormulas}{\mathcal{L}}

\renewcommand{\true}{\top}
\renewcommand{\false}{\bot}

\newcommand{\relsym}[1]{\thinspace{#1}\thinspace}
\newcommand{\knows}[2]{#1\relsym{\mathsf{k}}#2}

\newcommand{\limp}{\rightarrow}
\newcommand{\lequiv}{\leftrightarrow}

\newcommand{\proves}[4]{#1\relsym{:_{#3}^{#4}}#2}

\newcommand{\decides}[4]{#1\relsym{{\veebar}_{#3}^{#4}}#2}

\newcommand{\aModalFrame}{\mathfrak{S}}

\newcommand{\LiPded}{\vdash_{\LiP}}

\usepackage{colortbl} 
\renewcommand{\preorder}[1]{<_{#1}}
\renewcommand{\proves}[4]{#1\relsym{::_{#3}^{#4}}#2}

\newcommand{\LiiP}{\mathrm{LiiP}}

\newcommand{\LiiPded}{\vdash_{\LiiP}}

\newcommand{\pproves}[4]{#1\relsym{:_{#3}^{#4}}#2}

\newcommand{\LDiiP}{\mathrm{LDiiP}}

\newcommand{\LDiiPded}{\vdash_{\LDiiP}}

\newcommand{\LDiiPdedBis}{\mathrel{{\dashv}{\vdash}_{\LDiiP}}}

\begin{document}
\title{Logic of Negation-Complete Interactive Proofs\thanks{%
	Work partially funded with 
		Grant AFR~894328 from the National Research Fund Luxembourg  
			cofunded under the Marie-Curie Actions of the European Commission (FP7-COFUND)
			.}\\
		(Formal Theory of Epistemic Deciders)}
\author{Simon Kramer\\[\jot]
		\texttt{simon.kramer@a3.epfl.ch}}
\maketitle


\begin{abstract}
	We produce a decidable classical normal modal logic of 
		internalised 
			\emph{negation-complete} and thus \emph{disjunctive} 
				non-monotonic interactive proofs (LDiiP) from 
					an existing logical counterpart of 
						non-monotonic or instant interactive proofs (LiiP).
	LDiiP internalises agent-centric proof theories that 
		are 
			negation-complete (maximal) and 
			consistent (and hence strictly weaker than, for example, Peano Arithmetic) and
		enjoy the \emph{disjunction property} (like Intuitionistic Logic).
	In other words, 
		internalised proof theories are \emph{ultrafilters} and 
		all internalised proof goals are \emph{definite} in the sense of 
			being either provable or disprovable to an agent by means of disjunctive internalised proofs
				(thus also called \emph{epistemic deciders}).
	Still, LDiiP itself is 
		classical (monotonic, non-constructive), 
		negation-incomplete, and
		does not have the disjunction property.
	The price to pay for the negation completeness of our interactive proofs is their 
		non-monotonicity and 
		non-communality (for singleton agent communities only).
	As a normal modal logic, LDiiP enjoys a standard Kripke-semantics, which
		we justify by invoking the Axiom of Choice on LiiP's and 
	then construct in terms of a concrete oracle-computable function.
	LDiiP's agent-centric internalised notion of proof can also be viewed as 
		a \emph{negation-complete disjunctive explicit} refinement of standard KD45-belief, and
		yields a \emph{disjunctive} but negation-\emph{in}complete \emph{explicit} refinement of S4-provability.

	\bigskip
	
	\noindent
	\textbf{Keywords:} 
		agents as proof checkers, 
		constructive Kripke-semantics, 
		disjunctive explicit doxastic \& epistemic logic,
		epistemic deciders as decisive evidence, 
		interactive \& oracle computation,
		multi-agent systems, 
		negation as failure,  
		proofs as sufficient evidence, 
		proof terms as truth values.
\end{abstract}

\section{Introduction}
The subject matter of this paper is classical normal modal logic of non-monotonic interactive proofs, \ie 
	a \emph{novel} modal logic of \emph{negation-complete} and thus \emph{disjunctive} interactive proofs (LDiiP) and 
	an existing modal logic of 
		non-disjunctive and thus negation-incomplete interactive proofs (LiiP) 
			\cite{KramerICLA2013,LiiP}. 
(We abbreviate interactivity-related adjectives with lower-case letters.) 

Our goal here is to produce LDiiP axiomatically as well as semantically from LiiP.
Note that like in \cite{KramerICLA2013,LiiP,LiP}, 
	we still understand interactive \emph{proofs as sufficient evidence} for
		intended \emph{resource-unbounded} proof-checking agents (who are though unable to guess), and
		leave probabilistic and polynomial-time resource bounds for future work.

\subsection{Motivation}
Our immediate motivation for LDiiP is 
	first the theoretical concept and  
	second the practical application of 
a negation-complete variant of our interactive proofs \cite{KramerICLA2013,LiiP,LiP}.
The overarching motivation for LDiiP is 
	to serve in an intuitionistic foundation of interactive computation. 
See \cite{LiP} for a programmatic motivation.

\subsubsection{Theoretical concept}\label{section:TheoreticalConcept}
Like in the non-interactive setting of a single prover-verifier agent, 
	the motivation for negation-complete (maximal) and consistent logical theories 
		(or ultrafilters \cite{DaveyPriestley}) and their external and internalised notions of proof is 
			to gain cognitive, constructive, and computational content.

Recall 
	that a logical theory $\mathrm{T}$ is \emph{negation-complete} by definition if and only if 
		(written ``:iff'' hereafter)
		for all formulas $\phi$ in the language (say $\mathcal{L}$) of $\mathrm{T}$, 
			$\phi\in\mathrm{T}$ or $\neg\phi\in\mathrm{T}$, and
	that $\mathrm{T}$ is \emph{consistent} :iff 
		$\false\not\in\mathrm{T}$ (so $\mathrm{T}\neq\mathcal{L}$), where
			`$\neg$' designates negation (complementation) and
			$\false$ falsehood (bottom).
Notice that 
	each such logical theory  
		(a filter\footnote{A subset in a (logical) lattice is a \emph{filter} by definition if and only if
			it is closed under meet (conjunction) and the lattice ordering (implication) 
				\cite[Lindenbaum-Tarski algebra]{DaveyPriestley}.} of propositions)  
			is defined in terms of a characteristic property and thus independently of how it is generated 
				(\eg based on some proof system or satisfaction relation), and that 
	inconsistent theories are trivially negation-complete as well as classical.
Classic examples of 
	non-trivial negation-complete (first-order) theories (with equality, but without sets) are: 
		Tarski's fragment of Euclidean Geometry, 
		Presburger (natural-number) Arithmetic, and 
		elementary real-number arithmetic.
Given a recursive axiomatisation\footnote{
	I.e., $\mathrm{T}$ has an algorithmically decidable set of axioms.
		This is a minimal requirement for any practical logical theory;    
			it guarantees the recognizability of its axioms.} of and thus an \emph{external} notion of proof for $\mathrm{T}$, 
	negation completeness and 
	consistency corresponds to the meta-theorem schema 
		$\vdash_{\mathrm{T}}\phi$ or $\vdash_{\mathrm{T}}\neg\phi$ (NC) and 
		$\not\vdash_{\mathrm{T}}\false\,,$ respectively.
That is,
	for all $\phi\in\mathcal{L}$,
		$\phi$ or $\neg\phi$ is a theorem of $\mathrm{T}$, or, 
			model-theoretically speaking,	
				a validity, \ie a universal truth.
For negation-complete consistent \emph{modal} theories, 
	this incidentally means that 
		there is no local truth that is not also a global truth, and thus 
			the point of their modality (which is non-trivial local truth, \ie 
				truth in some but not all of their pointed models) is nullified.
(If $\vdash_{\mathrm{T}}\phi$ then $\phi$ is a universal and thus global truth;
	if $\not\vdash_{\mathrm{T}}\phi$ then $\vdash_{\mathrm{T}}\neg\phi$ by 
		the negation completeness of $\mathrm{T}$, and thus 
			$\neg\phi$ is a universal and thus global truth, and hence 
				$\phi$ cannot be a local truth by 
					the consistency of $\mathrm{T}$.)
So in some sense, negation-complete modal theories are trivial, even if they are consistent.
Fortunately here, our modal LDiiP is negation-\emph{in}complete.
It is only the notion of proof that LDiiP \emph{internalises} that is negation-complete.
Compared with LDiiP's internalised agent-centric notion of proof,
	negation completeness and 
	consistency corresponds to the axiom schema 
		$\LDiiPded(\decides{M}{\phi}{a}{})\lor(\decides{M}{\neg\phi}{a}{})$ and 
		$\LDiiPded\neg(\decides{M}{\false}{a}{})$, respectively, where
			$M$ designates a proof (message) and $a$ an intended proof-checking agent.
Notice how
	meta-logical negation and disjunction internalise as
		their object-logical counterparts.
Also, observe that 
	our internalisation is more concrete than its external counterpart in the sense that 
		the first speaks about a concrete (internalised) \emph{proof} (sufficient evidence) $M$ whereas 
		the latter only speaks about an abstract (external) prov\emph{ability} $\vdash_{\mathrm{T}}$.
Negation completeness means that 
	$M$ represents sufficient data (\eg a completion of the local system history recorded as a log file) for deciding  
		whether some statement (\eg about the current system state given by the global history) is true or false.
Hilbert hoped for a negation-complete consistent theory for the whole of mathematics, because, 
	in his word, there is no \emph{ignorabimus} in negation-complete consistent theories; 
	in some sense, they are cognitively ideal:
All (internalised) proof goals are \emph{definite} \cite{NotesOnSetTheory}, here in the sense that 
	their truth or falsehood can be determined unambiguously (and here even effectively by an agent) 
		by means of (internalised) proofs
			(thus also called \emph{epistemic deciders}).
Moreover, 
	negation-complete theories, though necessarily non-intuitionistic (!), nevertheless enjoy 
		the disjunction property of Intuitionistic Logic (IL),\footnote{See \cite{DisjunctionProperty} for
			a survey of other, so-called \emph{super-intuitionistic} or \emph{intermediate} logics strictly below classical propositional logic that
			also enjoy the disjunction property.} which is that 
			if $\vdash_{\mathrm{IL}}\phi\lor\phi'$ 
				then $\vdash_{\mathrm{IL}}\phi$ or $\vdash_{\mathrm{IL}}\phi'$ (DP) \cite{ConstructiveLogic}.
Thus they have considerable constructive content, and this even by 
	conserving the deductive convenience of the law of the excluded middle!
To see why negation-complete theories are necessarily classical, 
	suppose that there is a non-classical negation-complete theory $\mathrm{T}$ 
		(\ie $\not\vdash_{\mathrm{T}}\phi\lor\neg\phi$, and
				$\vdash_{\mathrm{T}}\phi$ or $\vdash_{\mathrm{T}}\neg\phi$) and
	derive an immediate contradiction therefrom by 
		considering the law of right and left $\lor$-introduction (set $\phi'\defeq\neg\phi$), which
			asserts that 
				if $\vdash_{\mathrm{T}}\phi$ or $\vdash_{\mathrm{T}}\phi'$ 
				then $\vdash_{\mathrm{T}}\phi\lor\phi'$ (and is also valid in IL).
In fact, 
	for classical logical theories, 
		negation completeness is classically equivalent to 
		the disjunction property.
This is a well-known result, which
	we recall here. 
\begin{theorem}\label{theorem:NCvsDP}
	For classical logical theories (filters in Boolean algebras or lattices), 
		negation completeness (maximality or being an ultrafilter) is classically equivalent to 
		the disjunction property (the property of being a prime filter).
\end{theorem}
\begin{proof}
	See Appendix~\ref{appendix:NCvsDP}.
\end{proof}
\noindent
Internalising negation-complete proof theories, 
	LDiiP thus internalises their disjunction property, as the theorem schema 
		$\LDiiPded(\decides{M}{(\phi\lor\phi')}{a}{})\limp((\decides{M}{\phi}{a}{})\lor(\decides{M}{\phi'}{a}{}))$,
			which is why we call our internalised proofs also \emph{disjunctive}.
Yet given  
	first, the classicality (and normality) of LDiiP, and 
	second, Theorem~\ref{theorem:NCvsDP}, which applies to the theories that LDiiP internalises, 
		we could as well have stipulated the internalised disjunction property as axiom schema and
		then derived the internalised negation completeness therefrom as theorem schema.
That is, in arbitrary classical normal modal logics, 
	we can make the following deduction, where
		the universal meta-quantification over $\phi$ and $\phi'$ in Line~1 is left implicit:
\begin{enumerate}
	\item$\vdash\Box(\phi\lor\phi')\limp(\Box\phi\lor\Box\phi')$\hfill assumed internalised disjunction property
	\item$\vdash\Box(\phi\lor\neg\phi)\limp(\Box\phi\lor\Box\neg\phi)$\hfill 
		1, particularisation (set $\phi'\defeq\neg\phi$)
	\item$\vdash\phi\lor\neg\phi$\hfill classical tautology
	\item$\vdash\Box(\phi\lor\neg\phi)$\hfill 3, necessitation (normality)
	\item$\vdash\Box\phi\lor\Box\neg\phi$\hfill 2, 4, \emph{modus ponens}. (internalised negation completeness)
\end{enumerate}
To see also the computational content in negation-complete consistent theories with a recursive axiomatisation as previously claimed, 
	recall from classical recursion theory \cite{LogicAndComplexity} that such theories 
		 are actually also recursive (algorithmically decidable) as a whole, \ie 
			not only in their set of axioms:
The recursiveness of the axioms of a theory implies
	the recursive enumerability of its theorems.
So in order to decide whether or not $\phi\in\mathrm{T}$ 
	for a given $\phi\in\mathcal{L}$ in the language $\mathcal{L}$ of such a theory $\mathrm{T}$,
			start the enumeration process of the members of $\mathrm{T}$.
				By the negation completeness of $\mathrm{T}$, either $\phi$ or $\neg\phi$ will pop up in the process.
				If $\phi$ pops up then stop, and 
					conclude that $\phi\in\mathrm{T}$;
				if $\neg\phi$ pops up then stop, and
					conclude that  
						$\phi\not\in\mathrm{T}$ by the consistency of $\mathrm{T}$.

In summary, 
	the cognitive, constructive, and computational content of 
		recursively axiomatised negation-complete consistent theories is distilled in their 
			maximal consistency, 
			disjunction property, and 
			algorithmic decidability, respectively.
However,
	their scope is far from the one of Hilbert's hope:
G\"odel ascertained the negation-\emph{in}completeness of 
	any recursively axiomatised consistent theory  
		containing the Peano-Arithmetic (PA) 
			part of mathematics \cite{LogicAndComplexity,IncompletenessSets}.\footnote{%
				Although 
					the natural numbers form a strict subset of the real numbers, 
						the negation-incomplete PA cannot be a subset of 
						the negation-complete elementary real-number arithmetic (R) mentioned earlier; 
							the natural numbers are not definable in the language of R \cite{FranzenI}.}
Worse, consistent theories containing PA are also algorithmically undecidable \cite{LogicAndComplexity}.
Notwithstanding,
		recursively axiomatised negation-complete consistent theories, which
			are thus strictly weaker than PA, 
				are crucial for practical applications.
(Maximally consistent sets are also crucial for theoretical applications such as
	the canonical-model construction for axiomatic completeness proofs, \cf Appendix~\ref{appendix:LDiiPCompleteness}.)

\subsubsection{Practical application}\label{section:PracticalApplication}
Both 
	the external as well as 
	the internalised form of negation completeness have important practical applications.
Important practical applications of 
	the external form ``\,$\vdash\phi$ or $\vdash\neg\phi$\,'' of negation completeness, which 
		have become classics in computer science and engineering, are 
	logic databases and programming.
There,
	the external form ``\,$\vdash\phi$ or $\vdash\neg\phi$\,'' classically corresponds to 
		the principle of \emph{negation as failure} ``\,$\not\vdash\phi$ implies $\vdash\neg\phi$\,'', \ie    
			$\neg\phi$ can be inferred if every possible proof of $\phi$ fails 
				\cite{NegationAsFailure,OnClosedWorldDatabases}.
Another important practical application of 
	a modal-logical variant 
		``\,$\not\vdash\mathsf{K}_{a}(\phi)$ implies $\vdash\neg\mathsf{K}_{a}(\phi)$\,'' of 
			negation as failure is artificial intelligence \cite{NonMonotonicKnowledge}, where
				$\mathsf{K}_{a}(\phi)$ reads as ``agent $a$ knows that $\phi$ (is true).''
There,
	this \emph{epistemic} variant of negation as failure produces a 
		\emph{non-monotonic} logic of knowledge for multi-agent distributed systems.
(This is also the only piece of related work that we are aware of.)
An important practical application of 
	our internalised form 
		$\LDiiPded(\decides{M}{\phi}{a}{})\lor\decides{M}{\neg\phi}{a}{}$ of negation completeness is 
		\emph{accountability} for dependable multi-agent distributed systems 
				(\eg electronic voting systems \cite{MMMSRTETEVVVS}, and,  
					more generally, the whole Internet \cite{AccountableInternet}).
A multi-agent distributed system $S$ is accountable by definition if and only if 
	$S$ is abuse-free and auditable \cite{MMFAAMAS}:
For all agents $b$ in $S$, 
		(\emph{abuse-freeness}), whenever $b$ behaves correctly (as an agent in $S$), 
			$b$ can prove to all agents $a$ (including to herself) in $S$ that she does so, and, 
		(\emph{auditability}), whenever $b$ behaves incorrectly (and thus is faulty), 
			every or at least one other agent $c$ in $S$ will eventually be able to prove to all agents $a$ in $S$ 
				(including to herself and $b$) that $b$ is faulty, 
					(\cf \cite{MMFAAMAS} for a formal transcription of this 
						natural-language formulation).
In such a system $S$, 
	each agent $b$'s behaviour in terms of 
		her past actions can be recorded in a \emph{log file} \cite{Logs} (say $M$) that 
			is broadcast; 
			and 
it is this log file $M$ that 
	must be constructed so as to have sufficient evidential strength to 
		constitute a negation-complete proof with respect to 
			the proof goal of $b$ behaving correctly (expressed with an atomic formula $\mathsf{correct}(b)$):
	$$(\decides{M}{\mathsf{correct}(b)}{a}{})\lor\decides{M}{\neg\,\mathsf{correct}(b)}{a}{}$$
In other words, $M$ must constitute 
	\emph{decisive evidence} or, in yet other words, be an \emph{epistemic decider} to $a$ about 
		the (ephemeral) issue of $b$'s correctness.
($b$ can change her behaviour!)
That is,
	\emph{LDiiP is a formal theory of epistemic deciders.}
For abuse-freeness (auditability), 
	the prover $b$ ($c$) must (eventually) know such an $M$, written $\knows{b}{M}$ ($\knows{c}{M}$).
We will present formal definitions in Section~\ref{section:LDiiP} and
	a full formal case study in future work (\cf \cite{MMFAAMAS} for a preliminary, 
		non-axiomatic accountability case study).
Finally, note that
	a piece of decisive evidence $M$ for $\mathsf{correct}(b)$ brought to the attention of a judge $a$ 
		can be viewed as a kind of \emph{forensic trace,} since 
			$M$ allows $a$ 
				to decide whether or not $b$ is correct and thus 
				to decide whether or not $b$ is guilty of behaving incorrectly.

\subsection{Contribution}\label{section:Contribution}
\paragraph{Conceptual contributions}
Our conceptual contributions in this paper are the following.
First, 
	we produce 
	a novel modal logic of negation-complete and thus disjunctive interactive proofs 
		(\cf Theorem~\ref{theorem:Adequacy}), which
		internalises agent-centric negation-complete consistent proof theories 
			(enjoying the disjunction property) and
				has important theoretical and practical applications.
Second, 
	we offer the insights that 
		the price to pay for negation completeness and disjunctiveness is 
			the non-monotonicity and non-communality of
				the resulting agent-centric notion of proof (\cf Fact~\ref{fact:NCiNM} and \ref{fact:NCiNC}, respectively), which turns out to be also 
					a negation-complete disjunctive explicit refinement of standard KD45-belief 
						(\cf Corollary~\ref{corollary:NCDEB}).
Third,
	we contribute 
		a disjunctive but negation-incomplete explicit refinement of S4-provability  
			(\cf Corollary~\ref{corollary:DEK}), constructed from 
				our notion of proof.

\paragraph{Technical contributions}
Our technical contributions are the following.
First, 
	we provide a standard but also oracle-computational and set-theoretically constructive Kripke-semantics for LDiiP  
		(\cf Section~\ref{section:Semantically}).
Like in \cite{KramerICLA2013,LiiP},
		we endow the proof modality with a standard Kripke-semantics \cite{ModalLogicSemanticPerspective}, but whose
			accessibility relation $\access{M}{a}{}$ we 
				first define constructively in terms of 
					elementary set-theoretic constructions,\footnote{in loose analogy with 
							the set-theoretically constructive rather than 
							the purely axiomatic definition of 
								numbers \cite{TheNumberSystems} of  
								ordered pairs (\eg the now standard definition by Kuratowski, 
									and other well-known definitions \cite{NotesOnSetTheory})} 
							namely as $\pAccess{M}{a}{}$,
				and then match to an abstract semantic interface in standard form (which 
					abstractly stipulates the characteristic properties of the accessibility relation
						\cite{ModalProofTheory}).
	We will say that $\pAccess{M}{a}{}$ \emph{exemplifies} 
		(or \emph{realises}) $\access{M}{a}{}$.
	(A simple example of a set-theoretically constructive but non-intuitionistic definition of a modal accessibility is
		the well-known definition of epistemic accessibility as 
			state indistinguishability defined in terms of  
				equality of state projections \cite{Epistemic_Logic}.)
	The Kripke-semantics for LDiiP is oracle-computational in the sense that
		(\cf Definition~\ref{definition:SemanticIngredients})
		the individual proof knowledge (say $M$) can be thought of as being provided by 
			an imaginary computation oracle, which thus acts as a 
				hypothetical provider and imaginary epistemic source of our interactive proofs.
Second, 
	we prove Theorem~\ref{theorem:ProofCompilability}, which 
		establishes 
			the proof-terms-as-truth-values view as well as  
			a normal form for the special case of a singleton agent universe.
Third,
	we prove 
		the finite-model property (\cf Theorem~\ref{theorem:FiniteModelProperty}) and  
		the algorithmic decidability of LDiiP (\cf  Corollary~\ref{corollary:AlgorithmicDecidability}).
(Negation completeness implies algorithmic decidability as seen in Section~\ref{section:TheoreticalConcept}, 
	but not vice versa as LDiiP testifies.)

\subsection{Roadmap}
In the next section, 
	we introduce our Logic of Disjunctive instant interactive Proofs (LDiiP) axiomatically by means of 
				a compact closure operator 
			that induces the Hilbert-style proof system that we seek. 
%
We then 
	gain the (syntactic) insight that negation completeness implies non-monotonicity 
		(\cf Fact~\ref{fact:NCiNM}), and 
	prove the above-mentioned  
		Theorem~\ref{theorem:ProofCompilability} as well as 
		Corollary~\ref{corollary:NCDEB} and \ref{corollary:DEK} within the obtained system.
Next, we 
	introduce the concretely constructed semantics 
	as well as the standard abstract semantic interface for LDiiP (\cf Section~\ref{section:Semantically}), and
	prove the axiomatic adequacy of the proof system with respect to this interface 
		(\cf Theorem~\ref{theorem:Adequacy}).
We justify the existence of the constructive semantics of LDiiP by 
	invoking the Axiom of Choice on LiiP's (\cf Table~\ref{table:LiiP}) and 
		then also construct it in terms of a concrete oracle-computable function, from which
			we gain the (semantic) insight that negation completeness implies non-communality 
				(\cf Fact~\ref{fact:NCiNC}).
Last but not least,
	we prove 
		the finite-model property 
			(\cf Theorem~\ref{theorem:FiniteModelProperty}) and, therefrom,  
		the algorithmic decidability 
			(\cf Corollary~\ref{corollary:AlgorithmicDecidability}) of LDiiP. 

\section{LDiiP}\label{section:LDiiP}
\subsection{Syntactically}
Like the Logic of instant interactive Proofs (LiiP), 
	the Logic of Disjunctive instant interactive Proofs (LDiiP) provides
		a modal \emph{formula language} over a generic message \emph{term language}.
The formula language of LDiiP offers
		the propositional constructors, 
		a relational symbol `$\knows{}{}$' for constructing atomic propositions about individual knowledge (\eg $\knows{a}{M}$), and
		a modal constructor `$\decides{}{}{a}{}$' for propositions about proofs
			(\eg $\decides{M}{\phi}{a}{}$).
In brief, LDiiP is a minimal extension of classical propositional logic with an interactively generalised
	additional operator (the proof modality) and 
	proof-term language.
Note, the language of LDiiP is identical to the one of LiiP \cite{KramerICLA2013,LiiP} 
	modulo the proof-modality notation, which
		in LiiP is `$\proves{}{}{a}{\community}$', where
			$a$ acts as proof checker, like in LDiiP, and 
			$\community$ as $a$'s peer group, unlike in LDiiP (non-communality).
\begin{definition}[The language of LDiiP]\label{definition:LDiiPLanguage}
	Let
	\begin{itemize}
		\item $\agents\neq\emptyset$ designate a non-empty finite set of 
			\emph{agent names} $a$, $b$, $c$, \etc 
		\item $\messages$ 
			designate 
			a language of \emph{message terms} $M$ such that $a\in\messages$ 
		\item $\mathcal{P}$ designate a denumerable set of \emph{propositional variables} $P$ constrained such that
			for all $a\in\agents$ and $M\in\messages$, 
				$(\knows{a}{M})\in\mathcal{P}$ (for ``$a$ knows $M$'') is 
					a distinguished variable, \ie 
						an \emph{atomic proposition}, (for \emph{individual} knowledge)
						
			(So, for $a\in\agents$, $\knows{a}{\cdot}$ is a unary relational symbol.)
		\item $\pFormulas\ni\phi \bnfeq P \bnfor 
				\neg\phi \bnfor 
				\phi\land\phi \bnfor 
				\colorbox[gray]{0.75}{$\decides{M}{\phi}{a}{}$}$ 
				designate our language of \emph{logical formulas} $\phi$, where
				$\decides{M}{\phi}{a}{}$ reads 
					``$M$ \emph{can disjunctively prove that} $\phi$ to $a$'' in the sense that 
					``$M$ \emph{can prove whether or not} $\phi$ (is true) to $a$.''
	\end{itemize}
\end{definition}
\noindent
Note the following macro-definitions: 
	$\true \defeq \decides{a}{\knows{a}{a}}{a}{}$, 
	$\false \defeq \neg \true$, 
	$\phi \lor \phi' \defeq \neg (\neg \phi \land \neg \phi')$,
	$\phi \limp \phi' \defeq \neg \phi \lor \phi'$, and 
	$\phi \lequiv \phi' \defeq (\phi \limp \phi') \land (\phi' \limp \phi)$. 

Then, LDiiP has the following axiom and deduction-rule schemas, where 
grey-shading indicates the remaining essential differences to LiiP \cite{KramerICLA2013,LiiP}.
\begin{definition}[The axioms and deduction rules of $\LDiiP$]\label{definition:AxiomsRules}
Let
	\begin{itemize}
		\item $\Gamma_{0}$ designate an adequate set of axioms for classical propositional logic
		\item \colorbox[gray]{0.75}{$\Gamma_{1}$ designate some appropriate set of axioms for $\knows{a}{M}$}
		\item $\Gamma_{2} \defeq \Gamma_{0} \cup \Gamma_{1} \cup \{$
			\begin{itemize}
				\item \colorbox[gray]{0.75}{$\decides{M}{\knows{a}{M}}{a}{}$%
				\quad(self-knowledge)}
				\item $(\decides{M}{(\phi\limp\phi')}{a}{})\limp
				((\decides{M}{\phi}{a}{})\limp\decides{M}{\phi'}{a}{})$%
				\quad(Kripke's law, K)
				\item $(\decides{M}{\phi}{a}{})\limp(\knows{a}{M}\limp\phi)$%
				\quad(epistemic truthfulness)
				\item \colorbox[gray]{0.75}{$\neg(\decides{M}{\false}{a}{})$\quad(proof consistency)}
				\item \colorbox[gray]{0.75}{$(\decides{M}{\phi}{a}{})\lor\decides{M}{\neg\phi}{a}{}$\quad(negation completeness)}
				\} 
			\end{itemize}
			designate the axiom schemas of LDiiP.
	\end{itemize}
	Then, $\colorbox[gray]{0.75}{$\LDiiP\defeq\Clo{}{}(\emptyset)$}\defeq\bigcup_{n\in\mathbb{N}}\Clo{}{n}(\emptyset)$, where 
		for all $\Gamma\subseteq\pFormulas$:
		\begin{eqnarray*}
					\Clo{}{0}(\Gamma) &\defeq& \Gamma_{2}\cup\Gamma\\
					\Clo{}{n+1}(\Gamma) &\defeq& 
						\begin{array}[t]{@{}l@{}}
							\Clo{}{n}(\Gamma)\ \cup\\
							\setst{\phi'}{\set{\phi,\phi\limp\phi'}\subseteq\Clo{}{n}(\Gamma)}\cup
								\quad\text{(\emph{modus ponens,} MP)}\\
							\setst{\decides{M}{\phi}{a}{}}{\phi\in\Clo{}{n}(\Gamma)}\cup
								\quad\text{(necessitation, N).}
						\end{array}
				\end{eqnarray*}
		We call $\LDiiP$ a \emph{base theory,} and
		$\Clo{}{}(\Gamma)$ an \emph{LDiiP-theory} for any $\Gamma\subseteq\pFormulas$.
\end{definition}
\noindent
Notice the 
	logical order of LDiiP, which like LiiP's is, 
		due to propositions about (proofs of) propositions, \emph{higher-order propositional}.
From LiiP \cite{KramerICLA2013,LiiP},  
	we recall the discussions of 
		Kripke's law (K), 
		the law of epistemic truthfulness, and 
		the law of necessitation (N):
The key to the validity of K is that 
	we understand interactive proofs as sufficient evidence for
		intended resource-unbounded proof-checking agents (who are though still unable to guess).
Clearly for such agents,
	if 
		$M$ is sufficient evidence for $\phi\limp\phi'$ and $\phi$
	then so is $M$ for $\phi'$.
Then, the significance of epistemic truthfulness to interactivity is that 
	in truly distributed multi-agent systems, 
		not all proofs are known by all agents, \ie 
			agents are not omniscient with respect to messages. 
Otherwise, why communicate with each other?
So there being a proof does not imply knowledge of that proof. 
When an agent $a$ does not know the proof and 
the agent cannot generate the proof \emph{ex nihilo} herself by guessing it, 
only communication from a peer, who thus acts as an oracle, can entail the knowledge of the proof with $a$. 
%
%
Next, the justification for N is that
	in interactive settings,
		validities, and thus \emph{a fortiori} tautologies 
			(in the strict sense of validities of the propositional fragment), 
				are in some sense trivialities \cite{LiP}.
To see why,
	recall that
		modal validities are true in \emph{all} pointed models 
			(\cf Definition~\ref{definition:TruthValidity}), and thus
			not worth being communicated from one point to another in a given model, \eg 
				by means of specific interactive proofs.
(Nothing is logically more embarrassing than  
	talking in tautologies.)
Therefore, 
	validities deserve \emph{arbitrary} proofs.
What is worth being communicated are
	truths weaker than validities, namely 
		local truths in the standard model-theoretic sense 
			(\cf Definition~\ref{definition:TruthValidity}), which
				may not hold universally.
Otherwise why communicate with each other?
We continue to discuss the remaining, new axioms and rules.
As mentioned,  
	the message language $\messages$ of LDiiP is generic, and
thus $\knows{a}{M}$ will require axioms that
	are appropriate to the term structure of the chosen $M\in\messages$
		(such as those required for LiiP \cite{KramerICLA2013,LiiP}).
The validity of the axiom schema of self-knowledge is justified by oracle computation: 
		``if $a$ were to receive $M$, \eg from an oracle,  
			then $a$ would know $M$'' (\cf Definition~\ref{definition:SemanticIngredients}).
(The law of self-knowledge is also valid in LiiP, where
	it corresponds to the theorem [but not axiom] schema $\proves{M}{\knows{a}{M}}{a}{\emptyset}$.)
The axiom schema of 
	proof consistency and 
	negation completeness internalises 
		(external theory) 
			consistency and negation completeness, respectively 
				(\cf Section~\ref{section:TheoreticalConcept}).
Observe that
	internalised negation completeness is defined independently of the proof-term structure ($M$ is abstract), just as 
		(external) negation completeness of a logical theory is defined independently of its possible proof-system structure. 
However,
	this abstract definition is 
		an indirect, structural constraint: 
after all, not any proof-system structure generates a negation-complete theory.
\begin{proposition}[Hilbert-style proof system]\label{proposition:Hilbert}
	Let 
		\begin{itemize}
			\item $\Phi\LDiiPded\phi$ \text{:iff} if $\Phi\subseteq\LDiiP$ then $\phi\in\LDiiP$ 
			\item $\phi\LDiiPdedBis\phi'$ \text{:iff} $\set{\phi}\LDiiPded\phi'$ and $\set{\phi'}\LDiiPded\phi$
			\item $\LDiiPded\phi$ \text{:iff} $\emptyset\LDiiPded\phi.$
		\end{itemize}
		In other words, ${\LDiiPded}\subseteq\powerset{\pFormulas}\times\pFormulas$ is a \emph{system of closure conditions} in the sense of 
				\cite[Definition~3.7.4]{PracticalFoundationsOfMathematics}.
		For example:
			\begin{enumerate}
				\item for all axioms $\phi\in\Gamma_{2}$, $\LDiiPded\phi$
				\item for \emph{modus ponens}, $\set{\phi,\phi\limp\phi'}\LDiiPded\phi'$
				\item for necessitation, $\set{\phi}\LDiiPded\decides{M}{\phi}{a}{}$.
			\end{enumerate}
		(In the space-saving, horizontal Hilbert-notation ``$\Phi\LDiiPded\phi$'', 
			$\Phi$ is not a set of hypotheses but a set of premises, \cf  
				\emph{modus ponens} and necessitation.)
	Then $\LDiiPded$ can be viewed as being defined by 
		a $\Clo{}{}$-induced Hilbert-style proof system.
	In fact 
			${\Clo{}{}}:\powerset{\pFormulas}\rightarrow\powerset{\pFormulas}$ is a \emph{standard consequence operator,} \ie
				a \emph{substitution-invariant compact closure operator.}
\end{proposition}
\begin{proof}
	Like in \cite{LiP}.
	That a Hilbert-style proof system can be viewed as induced by 
		 a compact closure operator is well-known (\eg see \cite{WhatIsALogicalSystem});
	that $\Clo{}{}$ is indeed such an operator can be verified by 
		inspection of the inductive definition of $\Clo{}{}$; and
	substitution invariance follows from our definitional use of axiom \emph{schemas}.\footnote{%
		Alternatively to axiom schemas,
		we could have used 
			axioms together with an
			additional substitution-rule set
				$\setst{\sigma[\phi]}{\phi\in\Clo{}{n}(\Gamma)}$
		in the definiens of $\Clo{}{n+1}(\Gamma)$.}
\end{proof}

\begin{corollary}[Normality]\label{corollary:Normality}
		LDiiP is a normal modal logic.
\end{corollary}
\begin{proof}
	Jointly by  
		Kripke's law, 
		\emph{modus ponens},
		necessitation (these by definition), and
		substitution invariance (\cf Proposition~\ref{proposition:Hilbert}).
\end{proof}
Note that in LDiiP,
	an analog of the primitive LiiP-rule 
	$$\set{\knows{a}{M}\lequiv\knows{a}{M'}}\LiiPded
		(\proves{M'}{\phi}{a}{\community})\lequiv\proves{M}{\phi}{a}{\community}\quad
			\text{(see \cite{KramerICLA2013,LiiP})}$$
	would be invalid (because incompatible with negation completeness) and thus is not admitted in LDiiP.
\emph{A fortiori,}
	an analog of the stronger primitive LiP-rule
	$$\set{\knows{a}{M}\limp\knows{a}{M'}}\LiPded
		(\pproves{M'}{\phi}{a}{\community})\limp\pproves{M}{\phi}{a}{\community}\quad
			\text{(see \cite{KramerICLA2013,LiP})}$$
	by which proof monotonicity 
		$\LiPded(\pproves{M}{\phi}{a}{\community})\limp\pproves{\pair{M}{M'}}{\phi}{a}{\community}$ 
		under paired data $M'$ can be deduced,  
			would be invalid and thus is not admitted in LDiiP either.
We thus assert the following negative fact about our negation-complete proofs.
\begin{fact}\label{fact:NCiNM}
	\emph{Negation completeness implies non-monotonicity.}
\end{fact}
\noindent
Note that 
	if we introduced a pairing constructor for proof terms into the message language $\messages$ of LDiiP 
	(as with LiiP, \cf Table~\ref{table:LiiP}),
	Fact~\ref{fact:NCiNM} would mean that 
$$\not\LDiiPded(\decides{M}{\phi}{a}{})\limp\decides{\pair{M}{M'}}{\phi}{a}{}\,.$$

\begin{fact}\label{fact:ProofConsistency}\ 
\begin{enumerate}
	\item $\set{\phi\limp\phi'}\LDiiPded(\decides{M}{\phi}{a}{})\limp\decides{M}{\phi'}{a}{}$\quad(regularity)
	\item $\LDiiPded\neg(\decides{M}{\false}{a}{})\lequiv
		((\decides{M}{\phi}{a}{})\limp\neg(\decides{M}{\neg\phi}{a}{}))$
	\item $\LDiiPded(\decides{M}{\neg\phi}{a}{})\lequiv
			\decides{M}{(\phi\limp\false)}{a}{}$
\end{enumerate}
\end{fact}
\begin{proof}
	1 and 2 are well-known for necessity modalities in arbitrary normal modal logics.
	For 3, consider 
		that $\LDiiPded\neg\phi\lequiv(\phi\limp\false)$ 
			since $\neg\phi\lequiv(\phi\limp\false)$ is a classical tautology, and
		then deduce the conclusion by 1.
\end{proof}

\begin{lemma}\label{lemma:PRD}\ 
	\begin{enumerate} 
		\item $\LDiiPded\decides{M}{((\decides{M}{\phi}{a}{})\limp\phi)}{a}{}$\quad(self-proof of truthfulness)
		\item $\LDiiPded(\decides{M}{(\decides{M}{\phi}{a}{})}{a}{})\limp\decides{M}{\phi}{a}{}$\quad(proof density)
	\end{enumerate}
\end{lemma}
\begin{proof}
	See Appendix~\ref{appendix:ProofLemma} 
\end{proof}
\noindent
The laws of self-proof of truthfulness and proof density also hold in LiiP \cite{KramerICLA2013,LiiP}.
We continue to present the first important result about LDiiP.
\begin{theorem}[Proof terms as Truth values]\label{theorem:ProofCompilability}\ 
	\begin{enumerate}
		\item $\LDiiPded(\decides{M}{\neg\phi}{a}{})\lequiv
				\neg(\decides{M}{\phi}{a}{})$\quad(maximal consistency)
		\item $\LDiiPded(\decides{M}{(\phi\land\phi')}{a}{})\lequiv((\decides{M}{\phi}{a}{})\land\decides{M}{\phi'}{a}{})$\quad(proof conjunctions \emph{bis})
		\item $\LDiiPded(\decides{M}{(\phi\lor\phi')}{a}{})\lequiv
				((\decides{M}{\phi}{a}{})\lor\decides{M}{\phi'}{a}{})$\quad(IDP \emph{bis})
		\item $\LDiiPded(\decides{M}{(\phi\limp\phi')}{a}{})\lequiv
				((\decides{M}{\phi}{a}{})\limp\decides{M}{\phi'}{a}{})$\quad(K \emph{bis})
		\item $\LDiiPded(\decides{M}{(\phi\lequiv\phi')}{a}{})\lequiv
				((\decides{M}{\phi}{a}{})\lequiv\decides{M}{\phi'}{a}{})$\quad(Bi-K)				
		\item $\LDiiPded(\decides{M}{(\decides{M}{\phi}{a}{})}{a}{})\lequiv\decides{M}{\phi}{a}{}$\quad(modal idempotency)
		\item $\LDiiPded\knows{b}{M}\limp
				((\decides{M}{(\decides{M}{\phi}{a}{})}{b}{})\lequiv\decides{M}{\phi}{a}{})$\quad
					(modal idempotency \emph{bis}) 
	\end{enumerate}
\end{theorem}
\begin{proof}
	See Appendix~\ref{appendix:ProofTermsAsTruthValues} 
\end{proof}
\noindent
``IDP'' abbreviates ``Internalised Disjunction Property.''
The laws are enumerated in a (total) order that respects their respective proof prerequisites.
Notice that 
	Theorem~\ref{theorem:ProofCompilability}.2--\ref{theorem:ProofCompilability}.5 are \emph{modal distributivity} laws.
They assert that the proof modality of LDiiP is fully distributive over (binary) Boolean operators.
While the laws of proof conjunction \emph{bis} and modal idempotency also hold in LiiP \cite{KramerICLA2013,LiiP}, 
only the if-direction of the laws IDP \emph{bis} and K \emph{bis} hold in LiiP.
Notice also that 
	modal idempotency combines proof density (\cf Lemma~\ref{lemma:PRD}.2) and 
	proof transitivity (\cf Line~l of the proof of modal idempotency).
Like in LiiP and LiP,
	the key to the validity of modal idempotency is that 
		each agent (\eg $a$) can act herself as proof checker, 
			see \cite[Section~3.2.2]{LiP} for more details.
The law of modal idempotency \emph{bis} is a generalisation of modal idempotency.
Observe that 
	when $|\agents|=1$,
		Theorem~\ref{theorem:ProofCompilability} implies that 
			all occurrences of the proof modality in a compound LDiiP-formula can be compiled away in the sense that
				all these occurrences can be pushed in front of 
					possibly negated atomic sub-formulas (\ie literals) of the compound formula,
						with the axiom formula $\decides{M}{\knows{a}{M}}{a}{}$ acting as base case.
Hence in this case, 
	we can understand \emph{proof terms as truth-values} in the spirit of 
		a form of realizability interpretation of constructive logic \cite[Section~7.8]{Realizability}.
Otherwise, \ie 
	when $|\agents|>1$ (recall from Definition~\ref{definition:LDiiPLanguage} that $\agents\neq\emptyset$), 
		it is possible that 
			not all such occurrences in a compound formula can be compiled away
				(\cf Theorem~\ref{theorem:ProofCompilability}.7).

The following corollary asserts that 
	our negation-complete and thus disjunctive proof modality is also 
		an \emph{explicit refinement} of the standard (implicit) belief modality \cite{MultiAgents}.
\begin{corollary}[Negation-complete Disjunctive Explicit Belief]\label{corollary:NCDEB}\ 
	`$\decides{M}{\cdot}{a}{}$' is 
				a \emph{negation-complete disjunctive KD45-modality} of explicit agent belief,
	where 
		$M$ represents the explicit evidence term that can justify agent $a$'s belief.
\end{corollary}
\begin{proof}
	Consider that 
		`$\decides{M}{\cdot}{a}{}$' satisfies 
			Kripke's law (K, \cf Definition~\ref{definition:AxiomsRules}), 
			the D-law (called ``proof consistency'' in Definition~\ref{definition:AxiomsRules}),
			the \textbf{4}-law (\cf the only-if part of Theorem~\ref{theorem:ProofCompilability}.6), 
			necessitation (\cf Definition~\ref{definition:AxiomsRules}), and
			negation completeness (\cf Definition~\ref{definition:AxiomsRules}), and
			thus the internalised disjunction property 
				(\cf the if-part of Theorem~\ref{theorem:ProofCompilability}.3).
	That `$\decides{M}{\cdot}{a}{}$' also satisfies the \textbf{5}-law can be proved as follows:
	\begin{enumerate}
		\item $\LDiiPded\neg(\decides{M}{\phi}{a}{})\limp
								(\decides{M}{\neg\phi}{a}{}{})$\hfill 
				only-if-part of Theorem~\ref{theorem:ProofCompilability}.1
		\item $\LDiiPded(\decides{M}{\neg\phi}{a}{}{})\limp
							\decides{M}{(\decides{M}{\neg\phi}{a}{}{})}{a}{}$\hfill 
				only-if-part of Theorem~\ref{theorem:ProofCompilability}.6[$\neg\phi$]
		\item $\LDiiPded\neg(\decides{M}{\phi}{a}{})\limp
								\decides{M}{(\decides{M}{\neg\phi}{a}{}{})}{a}{}$\hfill 
										1, 2, transitivity of $\limp$
		\item $\LDiiPded(\decides{M}{\neg\phi}{a}{})\limp
							\neg(\decides{M}{\phi}{a}{})$\hfill 
				if-part of Theorem~\ref{theorem:ProofCompilability}.1
		\item $\LDiiPded(\decides{M}{(\decides{M}{\neg\phi}{a}{})}{a}{})\limp
				\decides{M}{\neg(\decides{M}{\phi}{a}{})}{a}{}$\hfill 4, regularity
		\item $\LDiiPded\neg(\decides{M}{\phi}{a}{})\limp
				\decides{M}{\neg(\decides{M}{\phi}{a}{})}{a}{}$\hfill 3, 5, transitivity of $\limp$.
	\end{enumerate}
\end{proof}
Thanks to epistemic truthfulness, 
	$\knows{a}{M}$ is a sufficient condition for `$\decides{M}{\cdot}{a}{}$' to
		behave like a standard S5-knowledge modality \cite{MultiAgents,Epistemic_Logic,EpistemicLogicFiveQuestions}, which 
			not only obeys the D-law but also the stronger T-law, in the sense that
				$$\LDiiPded\knows{a}{M}\limp(\underbrace{(\decides{M}{\phi}{a}{})\limp\phi}_{\text{T-law}}).$$

In the following corollary, 
	we construct also 
		a disjunctive but negation-\emph{in}complete explicit refinement of (implicit) S4-provability.
\begin{corollary}[Disjunctive Explicit Provability]\label{corollary:DEK}
	`$\knows{a}{M}\land\decides{M}{\cdot}{a}{}$' is 
		a \emph{disjunctive but negation-incomplete S4-modality} of explicit agent provability, where 
			$M$ represents the explicit evidence term that does justify agent $a$'s knowledge.
\end{corollary}
\begin{proof}
	By Corollary~\ref{corollary:NCDEB} and 
	the fact that 
		the truth law 
			$\LDiiPded(\knows{a}{M}\land\decides{M}{\phi}{a}{})\limp\phi$ for 
				the modality `$\knows{a}{M}\land\decides{M}{\cdot}{a}{}$' 
					is equivalent to 
						the law of epistemic truthfulness (\cf Definition~\ref{definition:AxiomsRules}). 
	Note that 
		although the modality `$\knows{a}{M}\land\decides{M}{\cdot}{a}{}$' is evidently disjunctive, \ie 
			$\LDiiPded(\knows{a}{M}\land\decides{M}{(\phi\lor\phi')}{a}{})\limp 
				((\knows{a}{M}\land\decides{M}{\phi}{a}{})\lor(\knows{a}{M}\land\decides{M}{\phi'}{a}{}))$,
					it is negation-\emph{in}complete in that 
	$\not\LDiiPded(\knows{a}{M}\land\decides{M}{\phi}{a}{})\lor(\knows{a}{M}\land\decides{M}{\neg\phi}{a}{})$,
	because $\not\LDiiPded\knows{a}{M}$, in turn because of 
		the arbitrariness of $\Gamma_{1}$ (\cf Definition~\ref{definition:AxiomsRules}). 
	Fixing $\Gamma_{1}$ so that 
		a resource-unbounded agent $a$ unable to guess  
			knows all messages $M$ could only make sense for $\agents=\set{a}$.
	Otherwise, \ie when all agents know all messages, why interact with each other?
\end{proof}

\subsection{Semantically}\label{section:Semantically}
We continue to 
	present the concretely constructed semantics 
	as well as the standard abstract semantic interface for LDiiP, and
	prove the axiomatic adequacy of the proof system with respect to this interface.
We justify the existence of the constructive semantics of LDiiP by 
	invoking the Axiom of Choice on LiiP's \cite{KramerICLA2013,LiiP} and 
		then also construct it in terms of a concrete oracle-computable function.

\subsubsection{Concretely}
The ingredients for the concrete semantics of LiiP, from which 
	we will construct the concrete semantics of LDiiP, 
		are displayed in Table~\ref{table:LiiP}.
Therefrom, 
	we will only need 
		a concrete instance of $\states$ and $\msgs{a}$, and 
		an abstract instance of $\clo{a}{s}$ as ingredients for LDiiP.
Observe there that 
	the concrete accessibility $\pAccess{M}{a}{\community}$ of LiiP is 
		a totally defined proper (non-functional) relation.
Yet we do need a concrete accessibility relation for LDiiP that is functional, because
	LDiiP's negation-completeness axiom corresponds to 
	the functionality property of such a relation.
(LDiiP's proof consistency axiom corresponds to the 
	totality property of such a relation.)
Fortunately, 
	the concrete accessibility $\pAccess{M}{a}{\community}$ of LiiP is totally defined, and
	so we know by the Axiom of Choice AC[$\pAccess{M}{a}{\community}$], which 
	we may thus apply to $\pAccess{M}{a}{\community}$, that 
		$\pAccess{M}{a}{\community}$ can be ``functionalised,'' that is \cite{NotesOnSetTheory}, 
			\begin{multline*}\tag{AC[$\pAccess{M}{a}{\community}$]}
			\text{$\underbrace{\text{for all $s\in\states$, 
					there is $s'\in\states$ such that 
						$s\pAccess{M}{a}{\community}s'$}}_{\text{$\pAccess{M}{a}{\community}$ is totally defined}}$ implies}\\
			\text{$\underbrace{\text{there is $f:\states\rightarrow\states$ such that 
					for all $s\in\states$, 
						$s\pAccess{M}{a}{\community}f(s)$}}_{\text{$\pAccess{M}{a}{\community}$ 
							can be ``functionalised''}}.$}
			\end{multline*}
Notice that 
	the Axiom of Choice is non-constructive in that 
		it abstractly asserts the conditional existence of a certain $f$ but 
			without actually providing a concrete example of such an $f$.
Thus our problem now is to find such an $f$ for $\pAccess{M}{a}{\community}$, which
	will allow us to construct a functional concrete accessibility for LDiiP.
In Definition~\ref{definition:SemanticIngredients}, 
	we construct such an $f$ as an oracle-computational function $\sigma_{a}^{M}$ on concrete states 
	constructed inductively in terms of certain generalised successor functions.
The essential differences in Definition~\ref{definition:SemanticIngredients} to Table~\ref{table:LiiP} are 
grey-shaded.

\begin{table}
\centering
\caption{Semantic ingredients for LiiP \cite{KramerICLA2013,LiiP} (partially reused here for LDiiP)}
\smallskip\setlength{\fboxsep}{10pt}
\fbox{\begin{minipage}{0.94\textwidth}\small Let 
\begin{itemize}
	\item $\states\ni s$ designate the \emph{state space}---a set of \emph{system states} $s$
	\item $\msgs{a}:\states\rightarrow\powerset{\messages}$ designate 
			a \emph{raw-data extractor} that 
				extracts (without analysing) the (finite) set of messages from a system state $s$ that
					agent $a\in\agents$ has 
						either generated (assuming that only $a$ can generate $a$'s signature) or else 
						received \emph{as such} (not only as a strict subterm of another message)\label{page:RawData}; that is, $\msgs{a}(s)$ is $a$'s \emph{data base} in $s$
		\item $\clo{a}{s}:\powerset{\messages}\rightarrow\powerset{\messages}$ designate a \emph{data-mining operator} such that \label{page:DataMining}
			$\clo{a}{s}(\data)\defeq\clo{a}{}(\msgs{a}(s)\cup\data)\defeq\bigcup_{n\in\mathbb{N}}\clo{a}{n}(\msgs{a}(s)\cup\data)$, where for all $\data\subseteq\messages$:
				\begin{eqnarray*}
					\clo{a}{0}(\data) &\defeq& \set{a}\cup\data\\
					\clo{a}{n+1}(\data) &\defeq& 
						\begin{array}[t]{@{}l@{}}
							\clo{a}{n}(\data)\ \cup\\
							\setst{\pair{M}{M'}}{\set{M,M'}\subseteq\clo{a}{n}(\data)}\cup\quad\text{(pairing)}\\
							\setst{M, M'}{\pair{M}{M'}\in\clo{a}{n}(\data)}\cup\quad\text{(unpairing)}\\
							\setst{\sign{M}{a}}{M\in\clo{a}{n}(\data)}\cup\quad\text{(\emph{personal} signature \emph{synthesis})}\\
							\setst{\pair{M}{b}}{\sign{M}{b}\in\clo{a}{n}(\data)}\quad\text{(\emph{universal} signature \emph{analysis})}
						\end{array}
				\end{eqnarray*}
	\item \parbox[t]{0.9325\textwidth}{${\preorder{a}^{M}}\subseteq\states\times\states$ designate a 
	\emph{data preorder} on states such that
		for all $s,s'\in\states$,
			$s\preorder{a}^{M}s'$ :iff $\clo{a}{s}(\set{M})=\clo{a}{s'}(\emptyset)$,
				were $M$ can be viewed as \emph{oracle input} in addition to  
					$a$'s \emph{individual-knowledge base} $\clo{a}{s}(\emptyset)$ (\cf also \cite[Section~2.2]{LiP})}
	\item \parbox[t]{0.9325\textwidth}{${\preorder{\community}^{M}}\defeq(\bigcup_{a\in\community}{\preorder{a}^{M}})^{++}$, where 
			`$^{++}$' designates the closure operation of so-called \emph{generalised transitivity}
				in the sense that 
					 ${\preorder{\community}^{M}}\circ{\preorder{\community}^{M'}}\subseteq{\preorder{\community}^{\pair{M}{M'}}}$}
	\item ${\indist{a}{}{}}\defeq{\preorder{a}^{a}}$ designate an equivalence relation of 
		\emph{state indistinguishability}
	\item ${\pAccess{M}{a}{\community}}\subseteq\states\times\states$ designate 
		a \emph{concretely constructed accessibility relation}---short, \emph{concrete accessibility}---for 
			the non-monotonic proof modality of LiiP such that for all $s,s'\in\states$,  
			\begin{eqnarray*}
				s\pAccess{M}{a}{\community}s' &\defiff& 
								s'\in\hspace{-6ex} 
					\bigcup_{\scriptsize 
					\begin{array}{@{}c@{}}
						\text{$s\preorder{\community\cup\set{a}}^{M}\tilde{s}$ and
						}\\[0.5\jot] 
						M\in\clo{a}{\tilde{s}}(\emptyset) 
					\end{array}
					}\hspace{-5.5ex}
					[\tilde{s}]_{\indist{a}{}{}}\\
				&\text{(iff}& 
					\text{there is $\tilde{s}\in\states$ \st 
							$s\preorder{\community\cup\set{a}}^{M}\tilde{s}$ and
							$M\in\clo{a}{\tilde{s}}(\emptyset)$ and
							$\indist{a}{\tilde{s}}{s'}$).}
			\end{eqnarray*}
\end{itemize}
\end{minipage}}
\label{table:LiiP}
\end{table}

\begin{definition}[Semantic ingredients]\label{definition:SemanticIngredients}
For the set-theoretically constructive, model-theoretic study of LDiiP let 
\begin{itemize}
	\item $\states\ni s\colorbox[gray]{0.75}{$\bnfeq\mathtt{0}\bnfor \gscc{a}{M}(s)$}$, where 
				$\mathtt{0}$ can be understood as a zero data point (representing an initial state  
					for example), and  
				$\gscc{a}{M}$ can be read as ``agent $a$ receives message $M$ 
					(for example from another agent acting as an oracle)''
	\item $\msgs{a}:\states\rightarrow\powerset{\messages}$ be such that 
			\begin{align*}
				\msgs{a}(\mathtt{0}) &\defeq \emptyset\\
				\msgs{a}(\gscc{b}{M}(s)) &\defeq 
					\begin{cases}
						\msgs{a}(s)\cup\set{M} & \text{if $a=b$, and}\\
						\msgs{a}(s) & \text{otherwise}
					\end{cases}
			\end{align*}
	\item $\clo{a}{}:\powerset{\messages}\rightarrow\powerset{\messages}$ designate a compact closure operator
			and define $\clo{a}{s}:\powerset{\messages}\rightarrow\powerset{\messages}$ such that
			$\clo{a}{s}(\data)\defeq\clo{a}{}(\msgs{a}(s)\cup\data)\defeq\bigcup_{n\in\mathbb{N}}\clo{a}{n}(\msgs{a}(s)\cup\data)$
	\item \colorbox[gray]{0.75}{$\sigma_{a}^{M}:\states\rightarrow\states$ be so that 
				$\sigma_{a}^{M}(s)\defeq\begin{cases}
						s & \text{if $M\in\clo{a}{s}(\emptyset)$, and}\\
						\gscc{a}{M}(s) & \text{otherwise (oracle input)}
					\end{cases}$}
	\item ${\pAccess{M}{a}{}}\subseteq\states\times\states$ designate 
		a \emph{\textbf{concretely constructed} accessibility relation}---short, \emph{\textbf{concrete} accessibility}---for 
			the negation-complete disjunctive proof modality such that for all $s,s'\in\states$,
			$$\colorbox[gray]{0.75}{$s\pAccess{M}{a}{}s'$ :iff $s'=\sigma_{a}^{M}(s)$.}$$
\end{itemize}
\end{definition}

\begin{fact}\ 
	\begin{enumerate}
		\item $\sigma_{a}^{M}$ (and thus $\pAccess{M}{a}{}$) is oracle-computable.
		\item If $\clo{a}{}$ is polynomial-time computable 
				then so is $\sigma_{a}^{M}$ (and thus $\pAccess{M}{a}{}$).
	\end{enumerate}
\end{fact}
\begin{proof}
	Clearly, if $\clo{a}{}$ is computable then $\sigma_{a}^{M}$ is computable, and
	similarly for 2.
\end{proof}
\noindent
In particular when $\clo{a}{}=\mathrm{id}_{\powerset{\messages}}$, that is,   
	when $\clo{a}{}$ is the identity function on $\powerset{\messages}$ 
		($a$ performs no data-mining operations),
		$\pAccess{M}{a}{}$ is polynomial-time computable.

\begin{fact}\label{fact:AccessibilityRelations}
	For $\sigma_{a}^{M}$, fix $\clo{a}{}$ as in Table~\ref{table:LiiP}.
	Then: 
	\begin{enumerate}
		\item for all $s\in\states$, $s\pAccess{M}{a}{\community}\sigma_{a}^{M}(s)\,;$
		\item ${\pAccess{M}{a}{}}\subseteq{\pAccess{M}{a}{\emptyset}}$
					(and ${\pAccess{M}{a}{\emptyset}}\subseteq{\pAccess{M}{a}{\community}}$ \cite{KramerICLA2013,LiiP}).
	\end{enumerate}
\end{fact}
\begin{proof}
	Fix $\clo{a}{}$ as in Table~\ref{table:LiiP}.
	For 1, consider that 
		$s\preorder{a}^{M}\sigma_{a}^{M}(s)$ and 
		thus $s\preorder{\community\cup\set{a}}^{M}\sigma_{a}^{M}(s)$, 
		$M\in\clo{a}{\sigma_{a}^{M}(s)}(\emptyset)$, and
		$\sigma_{a}^{M}(s)\equiv_{a}\sigma_{a}^{M}(s)$ in Table~\ref{table:LiiP}.	
	Hence there is $\tilde{s}\in\states$ such that
		$s\preorder{\community\cup\set{a}}^{M}\tilde{s}$ and 
		$M\in\clo{a}{\tilde{s}}(\emptyset)$ and 
		$\tilde{s}\equiv_{a}\sigma_{a}^{M}(s)$.
	(In reverse, 
		$\sigma_{a}^{M}$ can be used as a Skolem-function for 
			the existential quantifier in the previous statement and
				thus in the definiens of $\pAccess{M}{a}{\community}$ in Table~\ref{table:LiiP}.) 
	For 2, 
		inspect 1 and definitions.
\end{proof}
\noindent
Hence we have indeed found in $\sigma_{a}^{M}$ an instance of 
	an $f$ for $\pAccess{M}{a}{\community}$ whose existence AC[$\pAccess{M}{a}{\community}$] postulates and
	thus indeed constructed a functional totally defined sub-relation $\pAccess{M}{a}{}$ of $\pAccess{M}{a}{\community}$---from $\pAccess{M}{a}{\community}$ itself 
		(as a Skolemnisation of its definiens).
However notice that
	we have lost $\community$ in $\pAccess{M}{a}{}$ (non-communality), because
		$\sigma_{a}^{M}$ simply disregards $\community$.
This is the price for the functionality of $\pAccess{M}{a}{}$.
Actually,
	$\pAccess{M}{a}{}$ (for LDiiP) is a functional analog of  
	$\preorder{a}^{M}$ (for LiiP, see Table~\ref{table:LiiP}).
And it is impossible to construct a functional analog of 
$\pAccess{M}{a}{\community}$ from a union of $\pAccess{M}{a}{}$ over $\community$, because
	such a union of functions need not be a function anymore.
In contrast, it is possible to construct a functional analog of
$\pAccess{M}{a}{\community}$ from an intersection of $\pAccess{M}{a}{}$ over $\community$, since
	such an intersection of functions is again a function.
Yet unfortunately it then need not be total anymore!
We can thus assert the following negative fact about our negation-complete proofs.
\begin{fact}\label{fact:NCiNC}
	Negation-completeness implies non-communality.
\end{fact}
\noindent
This fact could be useful to establish the theoretical and thus also practical impossibility 
	of engineering social procedures \cite{SocialSoftwareBis} for which  
		negation completeness would be a necessary condition.
Due to the same fact, 
	there is no community parameter $\community$ in `$\decides{}{}{a}{}$' and, in particular, 
		no LDiiP-analog of the LiiP-axiom 
		$$\LiiPded(\proves{M}{\phi}{a}{\community\cup\community'})\limp
				\proves{M}{\phi}{a}{\community}\quad\text{(see \cite{KramerICLA2013,LiiP})}.$$
Note that if we were to mix LiiP- and LDiiP-modalities in a single logic,
	the formula 
		$(\proves{M}{\phi}{a}{\emptyset})\limp\decides{M}{\phi}{a}{}$ would be a sound axiom in that logic 
			due to Fact~\ref{fact:AccessibilityRelations}.2.

\begin{proposition}\label{proposition:ConcreteAccessibility}\  
\begin{enumerate}
	\item there is $s'\in\states$ such that $s\pAccess{M}{a}{}s'$\quad(seriality/totality)
	\item if $s\pAccess{M}{a}{}s'$ and 
				$s\pAccess{M}{a}{}s''$ then $s'=s''$\quad(determinism/functionality)
	\item if $M\in\clo{a}{s}(\emptyset)$ then 
			$s\pAccess{M}{a}{}s$\quad(conditional reflexivity)
	\item if $s\pAccess{M}{a}{}s'$ then $M\in\clo{a}{s'}(\emptyset)$ \quad(epistemic image)
\end{enumerate}
\end{proposition}
\begin{proof}
	By inspection of definitions.
	(For 4, consider that $M\in\clo{a}{\gscc{a}{M}(s)}(\emptyset)$.)
\end{proof}

\subsubsection{Abstractly}
We now continue to 
	present the abstract semantic interface for LDiiP, and 
	prove the axiomatic adequacy of the proof system with respect to this interface.

\begin{definition}[Kripke-model]\label{definition:KripkeModel}
We define the \emph{satisfaction relation} `\thinspace$\models$' for 
		$\LDiiP$ in Table~\ref{table:SatisfactionRelation}, 
	\begin{table}[t]
	\centering
	\caption{Satisfaction relation}
	\smallskip
	\fbox{$\begin{array}{@{}rcl@{}}
		(\aModalFrame, \mathcal{V}), s\models P &\text{:iff}& s\in\mathcal{V}(P)\\[\jot]
		(\aModalFrame, \mathcal{V}), s\models\neg\phi &\text{:iff}& \text{not $(\aModalFrame, \mathcal{V}), s\models\phi$}\\[\jot]
		(\aModalFrame, \mathcal{V}), s\models\phi\land\phi' &\text{:iff}& \text{$(\aModalFrame, \mathcal{V}), s\models\phi$ and $(\aModalFrame, \mathcal{V}), s\models\phi'$}\\[\jot]
		(\aModalFrame, \mathcal{V}), s\models\decides{M}{\phi}{a}{} &\text{:iff}& 
				\text{for all $s'\in\states$, 
						if $s\access{M}{a}{}s'$ 
						then $(\aModalFrame, \mathcal{V}), s'\models\phi$}
	\end{array}$}
	\label{table:SatisfactionRelation}
	\end{table}
where 
	\begin{itemize}
		\item $\mathcal{V}:\mathcal{P}\rightarrow\powerset{\states}$ designates a usual \emph{valuation function,} yet
			partially predefined such that for all $a\in\agents$ and $M\in\messages$,
						$$\mathcal{V}(\knows{a}{M})\defeq\setst{s\in\states}{M\in\clo{a}{s}(\emptyset)} $$
						for $\states$ assumed abstract (and thus general) like in Table~\ref{table:LiiP} and
								$\clo{a}{s}$ like in Definition~\ref{definition:SemanticIngredients} but
									with $\msgs{a}$ abstract (and thus general) like in Table~\ref{table:LiiP}
				
		\item $\aModalFrame\defeq(\states,\set{\access{M}{a}{}}_{M\in\messages,a\in\agents})$
			designates a (modal) \emph{frame} for $\LDiiP$ with 
		an \emph{\textbf{abstractly constrained} accessibility relation}---short, \emph{\textbf{abstract} accessibility}---${\access{M}{a}{}}\subseteq\states\times\states$ for 
			the negation-complete disjunctive proof modality such that---the \textbf{\emph{semantic interface:}}\label{page:AbstractProofAccessibility} 
			\begin{itemize}
				\item there is $s'\in\states$ such that $s\access{M}{a}{}s'$\quad(seriality/totality)
				\item if $s\access{M}{a}{}s'$ and 
						$s\access{M}{a}{}s''$ then $s'=s''$\quad(determinism/functionality)
				\item if $M\in\clo{a}{s}(\emptyset)$ then 
						$s\access{M}{a}{}s$\quad(conditional reflexivity)
				\item if $s\access{M}{a}{}s'$ then $M\in\clo{a}{s'}(\emptyset)$ \quad(epistemic image)
			\end{itemize}
		\item $(\aModalFrame,\mathcal{V})$ designates a (modal) \emph{model} for $\LDiiP$.
	\end{itemize}
\end{definition}
Looking back, 
	we recognise that Proposition~\ref{proposition:ConcreteAccessibility} actually establishes the important fact that
		our concrete accessibility $\pAccess{M}{a}{}$ in 
		Definition~\ref{definition:SemanticIngredients} realises 
			all the properties stipulated by  
		our abstract accessibility $\access{M}{a}{}$ in Definition~\ref{definition:KripkeModel};
		we say that 
		$$\text{$\pAccess{M}{a}{}$ \emph{exemplifies} (or \emph{realises}) $\access{M}{a}{}$.}$$

\begin{theorem}[Axiomatic adequacy]\label{theorem:Adequacy}\ 
	$\LDiiPded$ is \emph{adequate} for $\models$, \ie:
	\begin{enumerate}
		\item if $\LDiiPded\phi$ then $\models\phi$\quad(axiomatic soundness)
		\item if $\models\phi$ then $\LDiiPded\phi$\quad(semantic completeness).
	\end{enumerate}
\end{theorem}
\begin{proof}
	Both parts can be proved with standard means:
	soundness follows as usual from 
		the admissibility of the axioms and rules 
			(\cf Appendix~\ref{appendix:AxiomaticSoundness}); and 
	completeness follows by means of the classical construction of canonical models,
		using Lindenbaum's construction of maximally consistent sets 
			(\cf Appendix~\ref{appendix:LDiiPCompleteness}).
\end{proof}

\begin{theorem}[Finite-model property]\label{theorem:FiniteModelProperty}
	For any LDiiP-model $\mathfrak{M}$,
	if $\mathfrak{M}, s\models\phi$ 
	then there is a finite LDiiP-model $\mathfrak{M}_{\mathrm{fin}}$ such that 
		$\mathfrak{M}_{\mathrm{fin}}, s\models\phi$.
\end{theorem}
\begin{proof}
	By the fact that 
		the \emph{minimal filtration} \cite{ModalModelTheory}
			$$\mathfrak{M}_{\mathrm{flt}}^{\mathrm{min},\Gamma}\defeq
				(\states/_{\sim_{\Gamma}},
					\set{\access{M}{a}{\mathrm{min},\Gamma}}_{M\in\messages,a\in\agents},\mathcal{V}_{\Gamma})$$ of 
			any LDiiP-model $\mathfrak{M}\defeq(\states,\set{\access{M}{a}{}}_{M\in\messages,a\in\agents},\mathcal{V})$ 
			through a finite $\Gamma\subseteq\pFormulas$ is a finite LDiiP-model such that 
				for all $\gamma\in\Gamma$, 
					$\mathfrak{M},s\models\gamma$ if and only if 
					$\mathfrak{M}_{\mathrm{flt}}^{\mathrm{min},\Gamma},[s]_{\sim_{\Gamma}}\models\gamma$.
			Following \cite{ModalModelTheory} for our setting, 
				we define 
				\begin{eqnarray*}
					{\sim_{\Gamma}}&\defeq&
						\setst{(s,s')\in\states\times\states}{\text{for all $\gamma\in\Gamma$, 
							$\mathfrak{M},s\models\gamma$ iff $\mathfrak{M},s'\models\gamma$}}\\ 
					{\access{M}{a}{\mathrm{min},\Gamma}} &\defeq&
						\setst{([s]_{\sim_{\Gamma}},[s']_{\sim_{\Gamma}})}{(s,s')\in{\access{M}{a}{}}}\\
					\mathcal{V}_{\Gamma}(P)&\defeq&
							\setst{[s]_{\sim_{\Gamma}}}{s\in\mathcal{V}(P)}\,.
				\end{eqnarray*}
			We further fix 
				$M\in\clo{a}{[s]_{\sim_{\Gamma}}}(\emptyset)$ :iff 
				$[s]_{\sim_{\Gamma}}\in\mathcal{V}_{\Gamma}(\knows{a}{M})$, and
				choose $\Gamma$ to be the (finite) sub-formula closure of $\phi$.
			Hence, we are left to prove that 
				$\mathfrak{M}_{\mathrm{flt}}^{\mathrm{min},\Gamma}$ is indeed an LDiiP-model, which
				means that we are left to prove that ${\access{M}{a}{\mathrm{min},\Gamma}}$ has
					all the properties stipulated by the semantic interface of LDiiP:
			\begin{itemize}
				\item ${\access{M}{a}{\mathrm{min},\Gamma}}$ inherits 
					seriality/totality as well as determinism/functionality from ${\access{M}{a}{}}$, 
						as can be seen by inspecting the definition of ${\access{M}{a}{\mathrm{min},\Gamma}}$;
				\item for conditional reflexivity, suppose that
						$M\in\clo{a}{[s]_{\sim_{\Gamma}}}(\emptyset)$.
						Thus consecutively: 
							$[s]_{\sim_{\Gamma}}\in\mathcal{V}_{\Gamma}(\knows{a}{M})$ by definition,
							$s\in\mathcal{V}(\knows{a}{M})$ by definition,
							$M\in\clo{a}{s}(\emptyset)$ by definition,
							$s\access{M}{a}{}s$ by the conditional reflexivity of $\access{M}{a}{}$, and finally 
							$[s]_{\sim_{\Gamma}}\access{M}{a}{\mathrm{min},\Gamma}[s]_{\sim_{\Gamma}}$ by 
								definition;
				\item for the epistemic-image property, suppose that
						$[s]_{\sim_{\Gamma}}\access{M}{a}{\mathrm{min},\Gamma}[s']_{\sim_{\Gamma}}$.
						Thus consecutively: 
							$s\access{M}{a}{}s'$ by definition,
							$M\in\clo{a}{s'}(\emptyset)$ by the epistemic-image property of $\access{M}{a}{}$,
							$s'\in\mathcal{V}(\knows{a}{M})$ by definition,
							$[s']_{\sim_{\Gamma}}\in\mathcal{V}_{\Gamma}(\knows{a}{M})$ by definition, and finally
							$M\in\clo{a}{[s']_{\sim_{\Gamma}}}(\emptyset)$ by definition.
			\end{itemize}
\end{proof}

\begin{corollary}[Algorithmic decidability]\label{corollary:AlgorithmicDecidability}
	If the sub-theory generated by $\Gamma_{1}$ (\cf Definition~\ref{definition:AxiomsRules}) is algorithmically decidable 
	then LDiiP (over $\Gamma_{1}$) is so too.
\end{corollary}
\begin{proof}
	In order to algorithmically decide whether or not $\phi\in\LDiiP$ (that is, $\LDiiPded\phi$), 
		axiomatic adequacy allows us to check whether or not $\neg\phi$ is locally satisfiable (that is,
			whether or not $\mathfrak{M},s\models\neg\phi$ for 
				some LDiiP-model $\mathfrak{M}$ and state $s$;
					by assumption, $M\in\clo{a}{s}(\emptyset)$, modelling membership in the theory generated by $\Gamma_{1}$, is decidable.).
	But then, the finite-model property of LDiiP allows us 
		to enumerate all finite LDiiP-models $\mathfrak{M}_{\mathrm{fin}}$ up to a size of at most 2 to the power 
			of the size $n$ of the sub-formula closure of $\neg\phi$ and
		to check whether or not $\mathfrak{M}_{\mathrm{fin}},s\models\neg\phi$.
	(There are at most $2^n$ equivalence classes for $n$ formulas.)
\end{proof}
So in some sense,
	we have proved the algorithmic decidability of the epistemic decisiveness of the evidence terms in LDiiP.
Note that 
	the algorithmic complexity of LDiiP will depend on the specific choice of 
		$\Gamma_{1}$ in Definition~\ref{definition:AxiomsRules}.

\section{Conclusion}
We have produced LDiiP from LiiP with 
		as main contributions those described in Section~\ref{section:Contribution}. 
In future work,
	we shall work out 
	dynamic and first-order extensions of LDiiP as well as 
the preliminary case study \cite{MMFAAMAS} mentioned in Section~\ref{section:PracticalApplication}.

\bibliographystyle{alpha}

\appendix

\section{Remaining proofs}\label{appendix:Proofs}
\subsection{Proof of Theorem~\ref{theorem:NCvsDP}}\label{appendix:NCvsDP}
	Suppose that 
		$\mathrm{T}$ is a classical logical theory with language $\mathcal{L}$ 
			(\ie for all $\phi\in\mathcal{L}$, $\vdash_{\mathrm{T}}\phi\lor\neg\phi$).
	%
	%
	%
	%
	%
	\begin{itemize}
		\item For the if-direction,
		suppose that 
			for all $\phi\in\mathcal{L}$, 
				$\vdash_{\mathrm{T}}\phi$ or $\vdash_{\mathrm{T}}\neg\phi$, and
		let $\phi,\phi'\in\mathcal{L}$.
	Thus $\vdash_{\mathrm{T}}\phi$ or $\vdash_{\mathrm{T}}\neg\phi$.
	Let us proceed by case analysis of this disjunction: 
	\begin{itemize}
		\item So first suppose that $\vdash_{\mathrm{T}}\phi$. 
				Hence $\vdash_{\mathrm{T}}\phi$ or $\vdash_{\mathrm{T}}\phi'$ 
					(from \underline{A} infer \underline{A or B}), and
				thus $\vdash_{\mathrm{T}}\phi\lor\phi'$ (vacously) implies 
					$\vdash_{\mathrm{T}}\phi$ or $\vdash_{\mathrm{T}}\phi'$
						(from \underline{A or B} infer \underline{C implies A or B}).
		\item Now suppose that $\vdash_{\mathrm{T}}\neg\phi$. 
	%
	%
	%
	Further suppose that $\vdash_{\mathrm{T}}\phi\lor\phi'$ (that is, \underline{C}).
	Hence $\vdash_{\mathrm{T}}\phi'$ (that is, \underline{B}), and 
	thus $\vdash_{\mathrm{T}}\phi$ or $\vdash_{\mathrm{T}}\phi'$
		(from \underline{B} infer \underline{A or B}).
	(Thus inferring \underline{C implies A or B}.)
	\end{itemize}
		\item For the only-if direction, 
		suppose that 
			for all $\phi,\phi'\in\mathcal{L}$, 
				$\vdash_{\mathrm{T}}\phi\lor\phi'$ implies 
					$\vdash_{\mathrm{T}}\phi$ or $\vdash_{\mathrm{T}}\phi'$, and 
		let $\phi\in\mathcal{L}$.
	Hence $\vdash_{\mathrm{T}}\phi\lor\neg\phi$ implies 
					$\vdash_{\mathrm{T}}\phi$ or $\vdash_{\mathrm{T}}\neg\phi$ 
						(particularising the universally quantified $\phi'$ with $\neg\phi$).
	Hence $\vdash_{\mathrm{T}}\phi$ or $\vdash_{\mathrm{T}}\neg\phi$, since 
		we have initially supposed $\mathrm{T}$ to be classical. 
	\end{itemize}
	(See also \cite{DaveyPriestley}.)

\subsection{Proof of Lemma~\ref{lemma:PRD}}\label{appendix:ProofLemma}
	\begin{enumerate}
		\item \begin{enumerate}
				\item $\LDiiPded(\decides{M}{\phi}{a}{})\limp(\knows{a}{M}\limp\phi)$\hfill epistemic truthfulness
				\item $\LDiiPded\knows{a}{M}\limp((\decides{M}{\phi}{a}{})\limp\phi)$\hfill a, PL
				\item $\LDiiPded(\decides{M}{(\knows{a}{M})}{a}{})\limp
						\decides{M}{((\decides{M}{\phi}{a}{})\limp\phi)}{a}{}$\hfill b, regularity
				\item $\LDiiPded\decides{M}{\knows{a}{M}}{a}{}$\hfill self-knowledge
				\item $\LDiiPded\decides{M}{((\decides{M}{\phi}{a}{})\limp\phi)}{a}{}$\hfill c, d, PL.
		\end{enumerate}
		\item \begin{enumerate}
				\item $\LDiiPded\decides{M}{((\decides{M}{\phi}{a}{})\limp\phi)}{a}{}$\hfill 
						Lemma~\ref{lemma:PRD}.1
				\item $\LDiiPded(\decides{M}{((\decides{M}{\phi}{a}{})\limp\phi)}{a}{})\limp
						((\decides{M}{(\decides{M}{\phi}{a}{})}{a}{})\limp\decides{M}{\phi}{a}{})$\hfill K
				\item $\LDiiPded(\decides{M}{(\decides{M}{\phi}{a}{})}{a}{})\limp\decides{M}{\phi}{a}{}$\hfill a, b, PL.
			\end{enumerate}
	\end{enumerate}

\subsection{Proof of Theorem~\ref{theorem:ProofCompilability}}\label{appendix:ProofTermsAsTruthValues}
	\begin{enumerate}
		\item 	
			\begin{enumerate}
				\item $\LDiiPded\neg(\decides{M}{\false}{a}{})$\hfill proof consistency
				\item $\LDiiPded\neg(\decides{M}{\false}{a}{})\lequiv
		((\decides{M}{\phi}{a}{})\limp\neg(\decides{M}{\neg\phi}{a}{}))$\hfill Fact~\ref{fact:ProofConsistency}
				\item $\LDiiPded(\decides{M}{\phi}{a}{})\limp\neg(\decides{M}{\neg\phi}{a}{})$\hfill a, b, PL
				\item $\LDiiPded(\decides{M}{\neg\phi}{a}{})\limp\neg(\decides{M}{\phi}{a}{})$\hfill c, PL
				\item $\LDiiPded(\decides{M}{\phi}{a}{})\lor\decides{M}{\neg\phi}{a}{}$\hfill negation completeness 
				\item $\LDiiPded\neg(\decides{M}{\phi}{a}{})\limp\decides{M}{\neg\phi}{a}{}$\hfill e, PL
				\item $\LDiiPded(\decides{M}{\neg\phi}{a}{})\lequiv\neg(\decides{M}{\phi}{a}{})$\hfill d, f, PL.
			\end{enumerate}
		\item \begin{enumerate}
				\item $\LDiiPded\phi\limp(\phi'\limp(\phi\land\phi'))$\hfill tautology
				\item $\LDiiPded(\decides{M}{\phi}{a}{})\limp\decides{M}{(\phi'\limp(\phi\land\phi'))}{a}{}$\hfill a, regularity
				\item $\LDiiPded(\decides{M}{(\phi'\limp(\phi\land\phi'))}{a}{})\limp
									((\decides{M}{\phi'}{a}{})\limp\decides{M}{(\phi\land\phi'))}{a}{}$\hfill K
				\item $\LDiiPded(\decides{M}{\phi}{a}{})\limp((\decides{M}{\phi'}{a}{})\limp\decides{M}{(\phi\land\phi'))}{a}{}$\hfill b, c, PL
				\item $\LDiiPded((\decides{M}{\phi}{a}{})\land\decides{M}{\phi'}{a}{})\limp\decides{M}{(\phi\land\phi')}{a}{}$\hfill d, PL
				\item $\LDiiPded(\phi\land\phi')\limp\phi$\hfill tautology
				\item $\LDiiPded(\decides{M}{(\phi\land\phi')}{a}{})\limp\decides{M}{\phi}{a}{}$\hfill f, regularity
				\item $\LDiiPded(\phi\land\phi')\limp\phi'$\hfill tautology
				\item $\LDiiPded(\decides{M}{(\phi\land\phi')}{a}{})\limp\decides{M}{\phi'}{a}{}$\hfill h, regularity
				\item $\LDiiPded(\decides{M}{(\phi\land\phi')}{a}{})\limp
									((\decides{M}{\phi}{a}{})\land\decides{M}{\phi'}{a}{})$\hfill g, i, PL
				\item $\LDiiPded((\decides{M}{\phi}{a}{})\land\decides{M}{\phi'}{a}{})\lequiv\decides{M}{(\phi\land\phi')}{a}{}$\hfill e, j, PL.
			\end{enumerate}	
		\item \begin{enumerate}
				\item $\LDiiPded(\decides{M}{(\phi\lor\phi')}{a}{})\lequiv
						\decides{M}{\neg(\neg\phi\land\neg\phi')}{a}{}$\hfill definition
				\item $\LDiiPded(\decides{M}{\neg(\neg\phi\land\neg\phi')}{a}{})\lequiv
						\neg(\decides{M}{(\neg\phi\land\neg\phi')}{a}{})$\hfill Theorem~\ref{theorem:ProofCompilability}.1
				\item $\LDiiPded(\decides{M}{(\phi\lor\phi')}{a}{})\lequiv
							\neg(\decides{M}{(\neg\phi\land\neg\phi')}{a}{})$\hfill a, b, PL
				\item $\LDiiPded(\decides{M}{(\neg\phi\land\neg\phi')}{a}{})\lequiv
						((\decides{M}{\neg\phi}{a}{})\land
							 \decides{M}{\neg\phi'}{a}{})$\hfill Theorem~\ref{theorem:ProofCompilability}.2
				\item $\LDiiPded\neg(\decides{M}{(\neg\phi\land\neg\phi')}{a}{})\lequiv
						\neg((\decides{M}{\neg\phi}{a}{})\land
							 \decides{M}{\neg\phi'}{a}{})$\hfill d, PL
				\item $\LDiiPded(\decides{M}{(\phi\lor\phi')}{a}{})\lequiv
							\neg((\decides{M}{\neg\phi}{a}{})\land
							 \decides{M}{\neg\phi'}{a}{})$\hfill c, e, PL
				\item $\LDiiPded\neg((\decides{M}{\neg\phi}{a}{})\land
							 \decides{M}{\neg\phi'}{a}{})\lequiv
							 (\neg(\decides{M}{\neg\phi}{a}{})\lor
							 \neg(\decides{M}{\neg\phi'}{a}{}))$\hfill PL
				\item $\LDiiPded(\decides{M}{(\phi\lor\phi')}{a}{})\lequiv
							(\neg(\decides{M}{\neg\phi}{a}{})\lor
							 \neg(\decides{M}{\neg\phi'}{a}{}))$\hfill f, g, PL
				\item $\LDiiPded(\decides{M}{\neg\phi}{a}{})\lequiv
						\neg(\decides{M}{\phi}{a}{})$\hfill Theorem~\ref{theorem:ProofCompilability}.1
				\item $\LDiiPded\neg(\decides{M}{\neg\phi}{a}{})\lequiv
						(\decides{M}{\phi}{a}{})$\hfill i, PL
				\item $\LDiiPded(\decides{M}{\neg\phi'}{a}{})\lequiv
						\neg(\decides{M}{\phi'}{a}{})$\hfill Theorem~\ref{theorem:ProofCompilability}.1
				\item $\LDiiPded\neg(\decides{M}{\neg\phi'}{a}{})\lequiv
						(\decides{M}{\phi'}{a}{})$\hfill k, PL
				\item $\LDiiPded(\decides{M}{(\phi\lor\phi')}{a}{})\lequiv
							((\decides{M}{\phi}{a}{})\lor
							 \decides{M}{\phi'}{a}{})$\hfill h, j, l, PL.
			\end{enumerate}
		\item \begin{enumerate}
				\item $\LDiiPded((\decides{M}{\phi}{a}{})\limp\decides{M}{\phi'}{a}{})\lequiv(\neg(\decides{M}{\phi}{a}{})\lor\decides{M}{\phi'}{a}{})$\hfill definition
				\item $\LDiiPded(\decides{M}{\neg\phi}{a}{})\lequiv
						\neg(\decides{M}{\phi}{a}{})$\hfill Theorem~\ref{theorem:ProofCompilability}.1
				\item $\LDiiPded((\decides{M}{\phi}{a}{})\limp\decides{M}{\phi'}{a}{})\lequiv((\decides{M}{\neg\phi}{a}{})\lor\decides{M}{\phi'}{a}{})$\hfill a, b, PL
				\item $\LDiiPded(\decides{M}{(\neg\phi\lor\phi')}{a}{})\lequiv
							((\decides{M}{\neg\phi}{a}{})\lor
							 \decides{M}{\phi'}{a}{})$\hfill Theorem~\ref{theorem:ProofCompilability}.3
				\item $\LDiiPded((\decides{M}{\phi}{a}{})\limp\decides{M}{\phi'}{a}{})\lequiv\decides{M}{(\neg\phi\lor\phi')}{a}{}$\hfill c, d, PL 
				\item $\LDiiPded((\decides{M}{\phi}{a}{})\limp\decides{M}{\phi'}{a}{})\lequiv\decides{M}{(\phi\limp\phi')}{a}{}$\hfill e, definition. 
				\end{enumerate}
			\item by Theorem~\ref{theorem:ProofCompilability}.2 and \ref{theorem:ProofCompilability}.4. 
			\item \begin{enumerate}
					\item $\LDiiPded(\decides{M}{(\decides{M}{\phi}{a}{})}{a}{})\limp\decides{M}{\phi}{a}{}$\hfill Lemma~\ref{lemma:PRD}.2
					\item $\LDiiPded(\decides{M}{(\decides{M}{\neg\phi}{a}{})}{a}{})\limp\decides{M}{\neg\phi}{a}{}$\hfill Lemma~\ref{lemma:PRD}.2
					\item $\LDiiPded\neg(\decides{M}{\neg\phi}{a}{})\limp\neg
									(\decides{M}{(\decides{M}{\neg\phi}{a}{})}{a}{})$\hfill b, PL
					\item $\LDiiPded(\decides{M}{\neg\phi}{a}{})\lequiv
								\neg(\decides{M}{\phi}{a}{})$\hfill Theorem~\ref{theorem:ProofCompilability}.1
					\item $\LDiiPded\neg(\decides{M}{\neg\phi}{a}{})\lequiv
								(\decides{M}{\phi}{a}{})$\hfill d, PL
					\item $\LDiiPded(\decides{M}{\phi}{a}{})\limp\neg
									(\decides{M}{(\decides{M}{\neg\phi}{a}{})}{a}{})$\hfill c, e, PL
					\item $\LDiiPded(\decides{M}{(\decides{M}{\neg\phi}{a}{})}{a}{})\lequiv
								\decides{M}{\neg(\decides{M}{\phi}{a}{})}{a}{}$\hfill d, regularity
					\item $\LDiiPded\neg(\decides{M}{(\decides{M}{\neg\phi}{a}{})}{a}{})\lequiv
								\neg(\decides{M}{\neg(\decides{M}{\phi}{a}{})}{a}{})$\hfill g, PL
					\item $\LDiiPded(\decides{M}{\phi}{a}{})\limp\neg
									(\decides{M}{\neg(\decides{M}{\phi}{a}{})}{a}{})$\hfill f, h, PL
					\item $\LDiiPded(\decides{M}{\neg(\decides{M}{\phi}{a}{})}{a}{})\lequiv
								\neg(\decides{M}{(\decides{M}{\phi}{a}{})}{a}{})$\hfill Theorem~\ref{theorem:ProofCompilability}.1
					\item $\LDiiPded\neg(\decides{M}{\neg(\decides{M}{\phi}{a}{})}{a}{})\lequiv
								\decides{M}{(\decides{M}{\phi}{a}{})}{a}{}$\hfill j, PL
					\item $\LDiiPded(\decides{M}{\phi}{a}{})\limp
									\decides{M}{(\decides{M}{\phi}{a}{})}{a}{}$\hfill i, k, PL; (proof transitivity)
					\item $\LDiiPded(\decides{M}{(\decides{M}{\phi}{a}{})}{a}{})\lequiv
									\decides{M}{\phi}{a}{}$\hfill a, l, PL.
				\end{enumerate}
		\item 	\begin{enumerate}
					\item $\LDiiPded\knows{b}{M}\limp 
						((\decides{M}{(\decides{M}{\phi}{a}{})}{b}{})\limp\decides{M}{\phi}{a}{})$\hfill
							epistemic truthfulness, PL
					\item $\LDiiPded\knows{b}{M}\limp 
						((\decides{M}{(\decides{M}{\neg\phi}{a}{})}{b}{})\limp\decides{M}{\neg\phi}{a}{})$\hfill
							dito a
					\item $\LDiiPded\knows{b}{M}\limp 
						(\neg(\decides{M}{\neg\phi}{a}{})\limp\neg(\decides{M}{(\decides{M}{\neg\phi}{a}{})}{b}{}))$\hfill b, PL
					\item $\LDiiPded(\decides{M}{\neg\phi}{a}{})\lequiv
								\neg(\decides{M}{\phi}{a}{})$\hfill Theorem~\ref{theorem:ProofCompilability}.1
					\item $\LDiiPded\neg(\decides{M}{\neg\phi}{a}{})\lequiv
								(\decides{M}{\phi}{a}{})$\hfill d, PL
					\item $\LDiiPded\knows{b}{M}\limp 
						((\decides{M}{\phi}{a}{})\limp\neg(\decides{M}{(\decides{M}{\neg\phi}{a}{})}{b}{}))$\hfill c, e, PL
					\item $\LDiiPded(\decides{M}{(\decides{M}{\neg\phi}{a}{})}{b}{})\lequiv
								\decides{M}{\neg(\decides{M}{\phi}{a}{})}{b}{}$\hfill d, regularity
					\item $\LDiiPded\neg(\decides{M}{(\decides{M}{\neg\phi}{a}{})}{b}{})\lequiv
								\neg(\decides{M}{\neg(\decides{M}{\phi}{a}{})}{b}{})$\hfill g, PL
					\item $\LDiiPded\knows{b}{M}\limp 
						((\decides{M}{\phi}{a}{})\limp\neg(\decides{M}{\neg(\decides{M}{\phi}{a}{})}{b}{}))$\hfill f, h, PL
					\item $\LDiiPded(\decides{M}{\neg(\decides{M}{\phi}{a}{})}{b}{})\lequiv
								\neg(\decides{M}{(\decides{M}{\phi}{a}{})}{b}{})$\hfill Theorem~\ref{theorem:ProofCompilability}.1
					\item $\LDiiPded\neg(\decides{M}{\neg(\decides{M}{\phi}{a}{})}{b}{})\lequiv
								\decides{M}{(\decides{M}{\phi}{a}{})}{b}{}$\hfill j, PL
					\item $\LDiiPded\knows{b}{M}\limp 
						((\decides{M}{\phi}{a}{})\limp\decides{M}{(\decides{M}{\phi}{a}{})}{b}{})$\hfill i, k, PL
					\item $\LDiiPded\knows{b}{M}\limp
						((\decides{M}{(\decides{M}{\phi}{a}{})}{b}{})\lequiv\decides{M}{\phi}{a}{})$\hfill
							a, l, PL.
		\end{enumerate}
	\end{enumerate}

\subsection{Proof of Theorem~\ref{theorem:Adequacy}}
	\newcommand{\canrel}[3]{\mathrel{_{#1}\negthinspace\mathrm{C}_{#2}^{#3}}}
	\newcommand{\canVal}{\mathcal{V}_{\mathsf{C}}}

\subsubsection{Axiomatic soundness}\label{appendix:AxiomaticSoundness}	
\begin{definition}[Truth \& Validity \cite{ModalLogicSemanticPerspective}]\label{definition:TruthValidity}\  
	\begin{itemize}
	\item The formula $\phi\in\pFormulas$ is \emph{true} (or \emph{satisfied}) 
		in the model $(\aModalFrame,\mathcal{V})$ at the state $s\in\states$ 
			:iff $(\aModalFrame,\mathcal{V}), s\models\phi$.
	\item The formula $\phi$ is \emph{satisfiable} in the model $(\aModalFrame,\mathcal{V})$ 
			:iff there is $s\in\states$ such that 
				$(\aModalFrame,\mathcal{V}), s\models\phi$.
	\item The formula $\phi$ is \emph{globally true} (or \emph{globally satisfied}) 
		in the model $(\aModalFrame,\mathcal{V})$, 
			written $(\aModalFrame,\mathcal{V})\models\phi$, :iff 
				for all $s\in\states$, $(\aModalFrame,\mathcal{V}),s\models\phi$.
	\item The formula $\phi$ is \emph{satisfiable}  
			:iff there is a model $(\aModalFrame,\mathcal{V})$ and a state $s\in\states$ such that 
				$(\aModalFrame,\mathcal{V}),s\models\phi$.	
	\item The formula $\phi$ is \emph{valid}, written $\models\phi$, :iff 
			for all models $(\aModalFrame,\mathcal{V})$, $(\aModalFrame,\mathcal{V})\models\phi$.
	\end{itemize}
\end{definition}

\begin{proposition}[Admissibility of LDiiP-specific axioms and rules]\label{proposition:AxiomAndRuleAdmissibility}\ 
	\begin{enumerate}
		\item $\models\decides{M}{\knows{a}{M}}{a}{}$
		\item $\models(\decides{M}{(\phi\limp\phi')}{a}{})\limp
				((\decides{M}{\phi}{a}{})\limp\decides{M}{\phi'}{a}{})$
		\item $\models(\decides{M}{\phi}{a}{})\limp(\knows{a}{M}\limp\phi)$
		\item $\models\neg(\decides{M}{\false}{a}{})$
		\item $\models(\decides{M}{\phi}{a}{})\lor\decides{M}{\neg\phi}{a}{}$
		\item If $\models\phi$ then $\models\decides{M}{\phi}{a}{}$
	\end{enumerate}
\end{proposition}
\begin{proof}
	1 follows directly from the epistemic-image property of $\access{M}{a}{}$;
	2 and 6 hold by the fact that LiiP has a standard Kripke-semantics;
	3 follows directly from the conditional reflexivity of $\access{M}{a}{}$, and
	4 and 5 from the seriality/totality and the determinism/functionality of $\access{M}{a}{}$, respectively.
\end{proof}

\subsubsection{Semantic completeness}\label{appendix:LDiiPCompleteness}	
	For all $\phi\in\pFormulas$, 
		if $\models\phi$ then $\LDiiPded\phi$.
\begin{proof}
	Let
		\begin{itemize}
			\item $\mathcal{W}$ designate the set of all maximally LDiiP-consistent sets\footnote{*
				A set $W$ of LDiiP-formulas is maximally LDiiP-consistent :iff 
					$W$ is LDiiP-consistent and 
					$W$ has no proper superset that is LDiiP-consistent.
				A set $W$ of LDiiP-formulas is LDiiP-consistent :iff 
					$W$ is not LDiiP-inconsistent.
				A set $W$ of LDiiP-formulas is LDiiP-inconsistent :iff 
					there is a finite $W'\subseteq W$ such that $((\bigwedge W')\limp\false)\in\LDiiP$.
				Any LDiiP-consistent set can be extended to a maximally LDiiP-consistent set by means of  
					the Lindenbaum Construction \cite[Page~90]{ModalProofTheory}.
				A set is maximally LDiiP-consistent if and only if 
				the set of logical-equivalence classes of the set is an ultrafilter of
				the Lindenbaum-Tarski algebra of LDiiP \cite[Page~351]{AlgebrasAndCoalgebras}.
				The canonical frame is isomorphic to the ultrafilter frame of that Lindenbaum-Tarski algebra 
					\cite[Page~352]{AlgebrasAndCoalgebras}.}
			\item for all $w,w'\in\mathcal{W}$,
				$w\canrel{M}{a}{}w'$ :iff $\setst{\phi\in\pFormulas}{\decides{M}{\phi}{a}{}\in w}\subseteq w'$
			\item for all $w\in\mathcal{W}$, $w\in\canVal(P)$ :iff $P\in w$.
		\end{itemize}
	Then \newcommand{\canModel}{\mathfrak{M}_{\mathsf{C}}}
		$\canModel\defeq
			(\mathcal{W},\set{\canrel{M}{a}{}}_{M\in\messages,a\in\agents},\canVal)$
		designates the \emph{canonical model} for LDiiP.
	Following Fitting \cite[Section~2.2]{ModalProofTheory}, 
	the following useful property of $\canModel$,  
		$$\boxed{$\text{for all $\phi\in\pFormulas$ and $w\in\mathcal{W}$,
			$\phi\in w$ if and only if $\canModel,w\models\phi$,}$}$$
	the so-called \emph{Truth Lemma}, can be proved by induction on the structure of $\phi$:
	\begin{enumerate}
		\item Base case ($\phi\defeq P$ for $P\in\mathcal{P}$). 
			For all $w\in\mathcal{W}$,
				$P\in w$ if and only if $\canModel,w\models P$, 
					by definition of $\canVal$.
		\item Inductive step ($\phi\defeq \neg\phi'$ for $\phi'\in\pFormulas$).
			Suppose that
				for all $w\in\mathcal{W}$,
					$\phi'\in w$ if and only if $\canModel,w\models\phi'$.
			Further let
				$w\in\mathcal{W}$.
			Then, 
				$\neg\phi'\in w$ if and only if $\phi'\not\in w$ --- $w$ is consistent ---
				if and only if $\canModel,w\not\models\phi'$ --- by the induction hypothesis ---
				if and only if $\canModel,w\models\neg\phi'$.
		\item Inductive step ($\phi\defeq \phi'\land\phi''$ for $\phi',\phi''\in\pFormulas$).
			Suppose that 
				for all $w\in\mathcal{W}$,
					$\phi'\in w$ if and only if $\canModel,w\models\phi'$, and that
				for all $w\in\mathcal{W}$,
					$\phi''\in w$ if and only if $\canModel,w\models\phi''$.
			Further let
				$w\in\mathcal{W}$.
			Then, 
				$\phi'\land\phi''\in w$ if and only if 
					($\phi'\in w$ and $\phi''\in w$), because $w$ is maximal.
			Now suppose that
				$\phi'\in w$ and $\phi''\in w$.
			Hence, 
				$\canModel,w\models\phi'$ and 
				$\canModel,w\models\phi''$, by the induction hypotheses, and 
			thus $\canModel,w\models\phi'\land\phi''$.
			Conversely, suppose that
				$\canModel,w\models\phi'\land\phi''$.
			Then,
				$\canModel,w\models\phi'$ and 
				$\canModel,w\models\phi''$.
			Hence,
				$\phi'\in w$ and $\phi''\in w$, by the induction hypotheses.
			Thus,
				($\phi'\in w$ and $\phi''\in w$) if and only if
				($\canModel,w\models\phi'$ and 
				$\canModel,w\models\phi''$).
			Whence
				$\phi'\land\phi''\in w$ if and only if
				($\canModel,w\models\phi'$ and 
				$\canModel,w\models\phi''$), by transitivity.
		\item Inductive step ($\phi\defeq\decides{M}{\phi'}{a}{}$ for 
				$M\in\messages$, $a\in\agents$, and $\phi'\in\pFormulas$).
			\begin{flushleft}
			\nn{4.1} for all $w\in\mathcal{W}$,
						$\phi'\in w$ if and only if $\canModel,w\models\phi'$\hfill ind.\ hyp.\\[\jot]
			\nn{4.2}\quad $w\in\mathcal{W}$\hfill hyp.\\[2\jot]
			\nn{4.3}\qquad $\decides{M}{\phi'}{a}{}\in w$\hfill hyp.\\[\jot]
			\nn{4.4}\qquad\quad	$w'\in\mathcal{W}$\hfill hyp.\\[\jot]
			\nn{4.5}\qquad\qquad $w\canrel{M}{a}{}w'$\hfill hyp.\\[\jot]
			\nn{4.6}\qquad\qquad $\setst{\phi''\in\pFormulas}{\decides{M}{\phi''}{a}{}\in w}\subseteq w'$\hfill 4.5\\[\jot]
			\nn{4.7}\qquad\qquad $\phi'\in\setst{\phi''\in\pFormulas}{\decides{M}{\phi''}{a}{}\in w}$\hfill 4.3, 4.6\\[\jot]
			\nn{4.8}\qquad\qquad $\phi'\in w'$\hfill 4.6, 4.7\\[\jot]
			\nn{4.9}\qquad\qquad $\canModel,w'\models\phi'$\hfill 4.1, 4.4, 4.8\\[\jot]
			\nn{4.10}\qquad\quad if $w\canrel{M}{a}{}w'$ then $\canModel,w'\models\phi'$\hfill 4.5--4.9\\[\jot]
			\nn{4.11}\qquad for all $w'\in\mathcal{W}$, 
					if $w\canrel{M}{a}{}w'$ 
					then $\canModel,w'\models\phi'$\hfill 4.4--4.10\\[\jot]
			\nn{4.12}\qquad $\canModel,w\models\decides{M}{\phi'}{a}{}$\hfill 4.11\\[2\jot]
			\nn{4.13}\qquad $\decides{M}{\phi'}{a}{}\not\in w$\hfill hyp.\\[\jot]
			\nn{4.14}\qquad\quad $\mathcal{F}=\setst{\phi''\in\pFormulas}{\decides{M}{\phi''}{a}{}\in w}\cup\set{\neg\phi'}$\hfill hyp.\\[\jot]
			\nn{4.15}\qquad\qquad $\mathcal{F}$ is LDiiP-inconsistent\hfill hyp.\\[\jot]
			\nn{4.16}\qquad\qquad there is $\set{\decides{M}{\phi_{1}}{a}{},\ldots,\decides{M}{\phi_{n}}{a}{}}\subseteq w$ such that\\
			\nn{}\qquad\qquad $\LDiiPded(\phi_{1}\land\ldots\land\phi_{n}\land\neg\phi')\limp\false$\hfill 4.14, 4.15\\[\jot]
			\nn{4.17}\qquad\qquad\quad $\set{\decides{M}{\phi_{1}}{a}{},\ldots,\decides{M}{\phi_{n}}{a}{}}\subseteq w$ and\\
			\nn{}\qquad\qquad\quad $\LDiiPded(\phi_{1}\land\ldots\land\phi_{n}\land\neg\phi')\limp\false$\hfill hyp.\\[\jot]
			\nn{4.18}\qquad\qquad\quad $\LDiiPded(\phi_{1}\land\ldots\land\phi_{n})\limp\phi'$\hfill 4.17\\[\jot]
			\nn{4.19}\qquad\qquad\quad $\LDiiPded(\decides{M}{(\phi_{1}\land\ldots\land\phi_{n})}{a}{})\limp\decides{M}{\phi'}{a}{}$\hfill 4.18, regularity\\[\jot]
			\nn{4.20}\qquad\qquad\quad $\LDiiPded((\decides{M}{\phi_{1}}{a}{})\land\ldots\land(\decides{M}{\phi_{n}}{a}{}))\limp\decides{M}{\phi'}{a}{}$\hfill 4.19\\[\jot]
			\nn{4.21}\qquad\qquad\quad $\decides{M}{\phi'}{a}{}\in w$\hfill 4.17, 4.20, $w$ is maximal\\[\jot]
			\nn{4.22}\qquad\qquad\quad false\hfill 4.13, 4.21\\[\jot]
			\nn{4.23}\qquad\qquad false\hfill 4.16, 4.17--4.22\\[\jot]
			\nn{4.24}\qquad\quad $\mathcal{F}$ is LDiiP-consistent\hfill 4.15--4.23\\[\jot]
			\nn{4.25}\qquad\quad there is $w'\supseteq\mathcal{F}$ \st $w'$ is maximally LDiiP-consistent\hfill 4.24\\[\jot]
			\nn{4.26}\qquad\qquad $\mathcal{F}\subseteq w'$ and $w'$ is maximally LDiiP-consistent\hfill hyp.\\[\jot]
			\nn{4.27}\qquad\qquad $\setst{\phi''\in\pFormulas}{\decides{M}{\phi''}{a}{}\in w}\subseteq\mathcal{F}$\hfill 4.14\\[\jot]
			\nn{4.28}\qquad\qquad $\setst{\phi''\in\pFormulas}{\decides{M}{\phi''}{a}{}\in w}\subseteq w'$\hfill 4.26, 4.27\\[\jot]
			\nn{4.29}\qquad\qquad $w\canrel{M}{a}{}w'$\hfill 4.28\\[\jot]
			\nn{4.30}\qquad\qquad $w'\in\mathcal{W}$\hfill 4.26\\[\jot]
			\nn{4.31}\qquad\qquad $\neg\phi'\in\mathcal{F}$\hfill 4.14\\[\jot]
			\nn{4.32}\qquad\qquad $\neg\phi'\in w'$\hfill 4.26, 4.31\\[\jot]
			\nn{4.33}\qquad\qquad $\phi'\not\in w'$\hfill 4.26 ($w'$ is LDiiP-consistent), 4.32\\[\jot]
			\nn{4.34}\qquad\qquad $\canModel,w'\not\models\phi'$\hfill 4.1, 4.33\\[\jot]
			\nn{4.35}\qquad\qquad there is $w'\in\mathcal{W}$ \st 
				$w\canrel{M}{a}{}w'$ and $\canModel,w'\not\models\phi'$\hfill 4.29, 4.34\\[\jot]
			\nn{4.36}\qquad\qquad $\canModel,w\not\models\decides{M}{\phi'}{a}{}$\hfill 4.35\\[\jot]
			\nn{4.37}\qquad\quad $\canModel,w\not\models\decides{M}{\phi'}{a}{}$\hfill 4.25, 4.26--4.36\\[\jot]
			\nn{4.38}\qquad $\canModel,w\not\models\decides{M}{\phi'}{a}{}$\hfill 4.14--4.37\\[2\jot]
			\nn{4.39}\quad $\decides{M}{\phi'}{a}{}\in w$ if and only if $\canModel,w\models\decides{M}{\phi'}{a}{}$\hfill 4.3--4.12, 4.13--4.38\\[\jot]
			\nn{4.40} for all $w\in\mathcal{W}$,
				$\decides{M}{\phi'}{a}{}\in w$ if and only if $\canModel,w\models\decides{M}{\phi'}{a}{}$\hfill 4.2--4.39\\[\jot]
		\end{flushleft}
	\end{enumerate}

	With the Truth Lemma we can now prove that 
			for all $\phi\in\pFormulas$,
				if $\not\LDiiPded\phi$ then $\not\models\phi$.
	Let 
		$\phi\in\pFormulas$, and 
	suppose that 
		$\not\LDiiPded\phi$.
	Thus, 
		$\set{\neg\phi}$ 
			is LDiiP-consistent, and 
			can be extended to a maximally LDiiP-consistent set $w$, \ie
				$\neg\phi\in w\in\mathcal{W}$.
	Hence 
		$\canModel,w\models\neg\phi$, by the Truth Lemma.
	Thus: 
		$\canModel,w\not\models\phi$,
		$\canModel\not\models\phi$, and
		$\not\models\phi$.
	That is,
			$\canModel$ is a 
				\emph{universal} (for \emph{all} $\phi\in\pFormulas$) 
				\emph{counter-model} (if $\phi$ is a non-theorem then $\canModel$ falsifies $\phi$).
	
	We are left to prove that 
		$\canModel$ is also an \emph{LDiiP-model}. 
	So let us instantiate our data mining operator $\clo{a}{}$ (\cf Page~\pageref{page:DataMining}) on $\mathcal{W}$  
		by letting for all $w\in\mathcal{W}$
			$$\msgs{a}(w)\defeq\setst{M}{\knows{a}{M}\in w},$$ and let us prove that:
					\begin{enumerate}
						\item there is $w'\in\mathcal{W}$ such that $w\canrel{M}{a}{}w'$
						\item if $w\canrel{M}{a}{}w'$ and 
								$w\canrel{M}{a}{}w''$ then $w'=w''$
						\item if $M\in\clo{a}{w}(\emptyset)$ then 
								$w\canrel{M}{a}{}w$
						\item if $w\canrel{M}{a}{}w'$ then $M\in\clo{a}{w'}(\emptyset)$.
					\end{enumerate} 

	For (1),
		let $w\in\mathcal{W}$ and $\phi\in\pFormulas$, and
		suppose that $\decides{M}{\phi}{a}{}\in w$.
	For the sake of deriving the contrary, 
		further suppose that $\phi\not\in w$.
	Hence $\neg\phi\in w$ because $w$ is maximal, and 
	thus $\phi\limp\false\in w$. 
	Hence $(\decides{M}{\phi}{a}{})\limp\decides{M}{\false}{a}{}\in w$ by regularity.
	Hence $\decides{M}{\false}{a}{}\in w$ by the first supposition and \emph{modus ponens}.
	Hence $\neg(\decides{M}{\false}{a}{})\not\in w$ because $w$ is consistent.
	Yet since $w$ is maximal,
		$\neg(\decides{M}{\false}{a}{})\in w$ (proof consistency).
	Contradiction.
	Hence $w$ is actually a $w'$ such that $\phi\in w'$.

	For (2), let us first prove the following, so-called Reflection Lemma:
		$$\text{$\decides{M}{\phi}{a}{}\not\in w$ if and only if 
			$\decides{M}{\neg\phi}{a}{}\in w$.}$$
	So suppose that 
	\begin{itemize}
		\item $\decides{M}{\phi}{a}{}\not\in w$. 
				Hence $\neg(\decides{M}{\phi}{a}{})\in w$ because $w$ is maximal.
				Since $w$ is maximal, 
					$\neg(\decides{M}{\phi}{a}{})\limp\decides{M}{\neg\phi}{a}{}\in w$ (negation completeness).
				Hence $\decides{M}{\neg\phi}{a}{}\in w$ by \emph{modus ponens}.
		\item $\decides{M}{\neg\phi}{a}{}\in w$. 
				Since $w$ is maximal, 
					$(\decides{M}{\neg\phi}{a}{})\limp\neg(\decides{M}{\neg\neg\phi}{a}{})\in w$ (proof consistency).
				Hence $\neg(\decides{M}{\neg\neg\phi}{a}{})\in w$ by \emph{modus ponens}.
				Since $w$ is maximal, 
					$\phi\limp\neg\neg\phi\in w$.
				Hence $(\decides{M}{\phi}{a}{})\limp\decides{M}{\neg\neg\phi}{a}{}\in w$ by regularity.
				Hence $\neg(\decides{M}{\neg\neg\phi}{a}{})\limp\neg(\decides{M}{\phi}{a}{})\in w$ by contraposition.
				Hence $\neg(\decides{M}{\phi}{a}{})\in w$ by \emph{modus ponens}.
				Hence $\decides{M}{\phi}{a}{}\not\in w$ because $w$ is consistent.
	\end{itemize}
	Now for (2), 
		let $w,w',w''\in\mathcal{W}$ and 
		suppose that $w\canrel{M}{a}{}w'$ and $w\canrel{M}{a}{}w''$.
	That is,	
		(for all $\phi\in\pFormulas$, 
			if $\decides{M}{\phi}{a}{}\in w$
			then $\phi\in w'$) and 
		(for all $\phi\in\pFormulas$, 
			if $\decides{M}{\phi}{a}{}\in w$
			then $\phi\in w''$).
	Now let $\phi\in\pFormulas$ and suppose that 
	\begin{itemize}
		\item $\phi\in w'$. 
			Hence $\neg\phi\not\in w'$ because $w$ is consistent.
			Hence $\decides{M}{\neg\phi}{a}{}\not\in w$ by 
				particularisation of the first supposition with $\neg\phi$ and \emph{modus tollens}.
			Hence $\decides{M}{\phi}{a}{}\in w$ by the Reflection Lemma.
			Hence $\phi\in w''$ by the second supposition and \emph{modus ponens}.
		\item $\phi\in w''$. Hence $\phi\in w'$---symmetrically.
	\end{itemize}

	For (3), 
		let $w\in\mathcal{W}$ and 
		suppose that 
			$M\in\clo{a}{w}(\emptyset)$.
	Hence $\knows{a}{M}\in w$ due to the maximality of $w$.
	Further suppose that $\decides{M}{\phi}{a}{}\in w$.
	Since $w$ is maximal,
		$$\text{$(\decides{M}{\phi}{a}{})\limp(\knows{a}{M}\limp\phi)\in w$\quad(epistemic truthfulness).}$$
	Hence, 
		$\knows{a}{M}\limp\phi\in w$, and
		$\phi\in w$, by consecutive  \emph{modus ponens.}
	
	For (4),
		let $w,w'\in\mathcal{W}$ and 
		suppose that $w\canrel{M}{a}{}w'$.
	That is,	
		for all $\phi\in\pFormulas$, 
			if $\decides{M}{\phi}{a}{}\in w$
			then $\phi\in w'$.
	Since $w$ is maximal,
		$$\text{$\decides{M}{\knows{a}{M}}{a}{}\in w$\quad(self-knowledge).}$$
	Hence 
		$\knows{a}{M}\in w'$ by particularisation of the supposition, and 
	thus 
		$M\in\clo{a}{w'}(\emptyset)$ by the definition of $\clo{a}{w'}$.

\end{proof}

\end{document}